\documentclass{amsart}
\usepackage{lipsum}
\usepackage{amsfonts,amssymb}
\usepackage{graphicx}
\usepackage{epstopdf}
\usepackage{graphicx,amssymb,algorithm,algorithmic,tikz,hyperref}
\usepackage{amsopn}
\newcommand{\vect}{\boldsymbol}
\renewcommand{\vec}{\boldsymbol}

\DeclareMathOperator{\Imag}{Im}

\DeclareMathOperator{\interior}{int}
\DeclareMathOperator{\dist}{dist}
\renewcommand{\Im}{\Imag}

\usepackage{amsthm}
\newtheorem{lemma}{Lemma}[section]
\newtheorem{corollary}{Corollary}[section]
\newtheorem{theorem}{Theorem}[section]
\theoremstyle{remark}
\newtheorem{remark}{Remark}[section]
\theoremstyle{definition}
\newtheorem{definition}{Definition}[section]

\title[Random Walks, Faber Polynomials, Power Methods]{Random Walks, Faber Polynomials and Accelerated Power Methods}

\author{Peter Cowal}
\email{cowalp@oregonstate.edu}
\address{Department of Mathematics, Oregon State University, Corvallis, OR}
\author{Nicholas F. Marshall}
\email{marsnich@oregonstate.edu}
\address{Department of Mathematics, Oregon State University, Corvallis, OR}
\author{Sara Pollock}
\email{s.pollock@ufl.edu}
\address{Department of Mathematics, University of Florida, Gainesville, FL}
\thanks{The third author's work was partially supported by NSF grant DMS 2045059 (CAREER)}


\begin{document}
\begin{abstract}
In this paper, we construct families of polynomials defined by recurrence relations related to mean-zero random walks. We show these families of polynomials can be used to approximate $z^n$ by a polynomial of degree $\sim \sqrt{n}$ in associated radially convex domains in the complex plane. Moreover, we show that the constructed families of polynomials have a useful rapid growth property and a connection to Faber polynomials.
Applications to iterative linear algebra are presented, including the development of arbitrary-order dynamic momentum power iteration methods suitable for classes of non-symmetric matrices.
\end{abstract}

\maketitle
\section{Introduction} 

Polynomial approximation plays an essential role in the design and analysis of iterative algorithms in numerical linear algebra.
An extrapolation scheme in the form of an iterative method that alternates between applying a matrix and taking linear combinations of the result corresponds to applying a polynomial of a matrix. 
Understanding the action of this polynomial on the eigenvalues of the matrix can allow us to understand the behavior of the algorithm.

Examples of classic algorithms that can be studied from a polynomial perspective include the Chebyshev iteration \cite{golub1961chebyshev}, conjugate gradient \cite{hestenes1952methods}, GMRES \cite{saad1986gmres}, the Lanczos algorithm \cite{lanczos1950iteration}, and the Arnoldi iteration \cite{golub2006arnoldi}. 
In this work, our motivation is twofold. First, we are motivated by the momentum power methods \cite{austin2024dynamically,cowal2025faber,xu2018accelerated}, which accelerate the power method using extra ``momentum'' terms, where in
\cite{austin2024dynamically,cowal2025faber}, the momentum parameters are determined dynamically, inspired by the extrapolation methods of
\cite{nigam2022simple,pollock2021extrapolating}. Second, we are motivated by the random walk perspective, see \eqref{eq:chebyshev-expectation} below, on the recurrence of Chebyshev polynomials, which was introduced in \cite{sachdeva2014faster}. This random walk perspective can be used to prove that $x^n$ is approximately a polynomial of degree $\sim \sqrt{n}$ on $[-1,1]$, see \eqref{eq:approxxn}, which was first established by \cite{newman1976approximation}. 
In this manuscript,
we develop general random walk constructions as a means to obtain efficient approximations of $z^n$ in regions in the complex plane. 

In previous work \cite{cowal2025faber}, we developed this construction
for a specific deltoid region of the complex plane by performing a detailed study of a family of polynomials
satisfying the recurrence $P_{n+1}(z) = (3/2)zP_n(z) - (1/2)P_{n-2}(z)$. The current work takes a more general approach, constructing a class of polynomial families whose recurrences are related to mean-zero random walks. We further develop theoretical results explaining various trade-offs in the construction of these recurrences, as well as results explaining why cusps appear in the regions in the complex plane associated with these recurrences. 

\subsection{Background}\label{subsec:bg}
To provide the underlying intuition for our approach, we briefly describe the random-walk perspective for Chebyshev polynomials introduced by Sachdeva and Vishnoi in
 \cite{sachdeva2014faster}.
The Chebyshev polynomials of the first kind satisfy $T_0(x) = 1$, $T_1(x) = x$, and the recurrence relation 
$$T_{n+1}(x) = 2xT_n(x) - T_{n-1}(x), \quad \text{for} \quad n \ge 1.
$$
This recurrence can equivalently be expressed as
\begin{equation} \label{eq:chebyshev-expectation}
xT_n(x) = \frac{T_{n+1}(x) + T_{n-1}(x)}{2} = \mathbb{E} T_{n + Y}(x),
\quad \text{for} \quad n \ge 1,
\end{equation}
where $Y$ is a random variable satisfying $P(Y = 1) = P(Y = -1) = 1/2$. If we define $T_n(x) := T_{|n|}(x)$ for $n < 0$, then we can iterate \eqref{eq:chebyshev-expectation} to deduce that
\begin{equation} \label{eq:chebrecurrence}
\mathbb{E} T_{Y_1 + Y_2 + \dots + Y_n}(x) = x^n\mathbb{E} T_0(x) = x^n,
\end{equation}
where $Y,Y_1, \dots, Y_n$ are i.i.d. random variables.  Since $Y_1,\ldots,Y_n$ are independent symmetric $\pm 1$ random variables, concentration results imply that 
$$
|Y_1 + Y_2 + \dots + Y_n| \le t \sqrt{n},
$$
with probability at least $1-2 e^{-t^2/2}$. Combining this concentration result with \eqref{eq:chebrecurrence}, the fact that $T_n$ is bounded by $1$ on $[-1,1]$, and the triangle inequality, one can show for $n \in \mathbb{N}$ and $t > 0$, that we have
\begin{equation} \label{eq:approxxn}
\left|x^n - \sum_{k=0}^{\lfloor t \sqrt{n} \rfloor} \alpha_k T_k(x) \right| \le 2 e^{-t^2/2}, \qquad \text{for} \quad x \in [-1,1],
\end{equation}
where the coefficients $\alpha_k = \mathbb{P}(|Y_1+\cdots + Y_n| = k)$, see \cite[Theorem 3.3]{sachdeva2014faster}. 
This result says that $x^n$ is approximately 
a polynomial of degree $\sim \sqrt{n}$ on $[-1,1]$.
A closely related property of the Chebyshev polynomials is the following rapid growth property: 
\begin{equation} \label{eq:chebgrowfast}
T_k(1 + \varepsilon) \ge \frac{1}{2} (1 + \sqrt{2\varepsilon})^k,
\end{equation}
for all $\varepsilon > 0$ and $k \in \mathbb{N}$. This rapid growth property can be viewed as a different expression of the same idea that a polynomial of degree $\sim \sqrt{n}$ can act like $x^n$ in the sense that $T_k(x)$ grows roughly like $x^n$ at $x = 1+\varepsilon$ when $k = \sqrt{n}$.

The polynomials introduced in this paper
generalize the probabilistic argument of \eqref{eq:chebyshev-expectation}-\eqref{eq:approxxn} to approximate $z^n$ by a polynomial of degree $\sim \sqrt{n}$ in regions in the complex plane. Note that the Chebyshev polynomial argument for \eqref{eq:approxxn}
does not immediately generalize to the complex plane because the Chebyshev polynomials $T_n(z)$ grow exponentially with $n$ at any point $z$ with a non-zero imaginary part. While our construction is motivated by generalizing \eqref{eq:approxxn}, our construction naturally gives rise to a rapid growth property similar to \eqref{eq:chebgrowfast}. We next make the notion of rapid growth precise. 
\begin{definition} \label{defn:rapidgrowth}
We say a family of polynomials $P_n(z)$, $n \ge 0$, that is normalized such that $P_n(1) =1$, has rapid growth past $1$ if for all $\varepsilon > 0$ and $n \in \mathbb{N}$, we have
\begin{equation} \label{eq:defrapidgrowth}
|P_n(1+\varepsilon)| \ge c \left(1 + c' \sqrt{\varepsilon} \right)^n,
\end{equation}
for some absolute constants $c,c' > 0$.
\end{definition}
For example, the Chebyshev polynomials of the first kind $T_n(x)$ have rapid growth at $1$ due to \eqref{eq:chebgrowfast}. In contrast, the monomial $x^n$ does not have rapid growth at $1$ since for any $c' > 0$, there exists arbitrarily small $\varepsilon > 0$ such that the function $(1+\varepsilon)^n$ grows slower than $(1 + c'\sqrt{\varepsilon})^n$.  The applications to iterative linear algebra considered in this paper are based on this rapid growth property. 

\subsection{Motivation}
Algorithms based on the Chebyshev polynomials have also been adapted to work on regions of the complex plane, for instance in \cite{calvetti1994adaptive,golub1961chebyshev,manteuffel1977tchebychev}, and often in the context of large-scale eigenvalue problems in conjunction with the Arnoldi iteration
\cite{Ho90,HCB90,saad84,saad2011numerical,sadkane93}, 
or the Davidson method 
\cite{ZhSa07}.
The Chebyshev polynomials have constant magnitude on elliptical level lines in the complex plane,  and much of the literature focuses on computational methods for determining the parameters of an approximately optimal ellipse,  thereby obtaining a reasonably efficient method.  However, the recurrence of the scaled and dilated Chebyshev polynomials no longer corresponds to a mean-zero probability distribution, which we show produces sub-optimal growth rates when $|z| > 1$. From another point of view, it was shown in \cite{FiFr91} that Chebyshev polynomials on the complex plane do not always feature the optimality properties they have on the real line. In this work, we demonstrate that polynomials based on higher-order recurrence relations can provide an advantage.
While our approach is inspired by a random walk perspective of \cite{sachdeva2014faster}, it also has a connection to Faber polynomials, see \cite{suetin1998series} for background.
We note that Faber polynomials have previously been considered in the context of numerical linear algebra, in both general \cite{eiermann1989semiiterative,eiermann1989hybrid,niethammer1983analysis} and specific \cite{ramos2021preconditioning} settings, and often to approximate functions of non-Hermitian matrices \cite{benzi2007decay,pozza2019inexact}.

\subsection{Main contributions}
The main contributions of this paper are as follows:

\begin{enumerate}
\item \emph{Construct families of random walk-based polynomials:}
We construct families of polynomials based on integer-valued mean-zero random variables which generalize the probabilistic constructions of \cite{cowal2025faber,sachdeva2014faster}.

\item \emph{Establish boundedness, rapid growth, and fast approximation properties:}
We prove boundedness and rapid growth properties of the constructed families of polynomials in Theorems \ref{thm:charcurve} and \ref{thm:boundedrapid}, and we use these polynomials to show that $z^n$ can be approximated by a polynomial of degree $\sim \sqrt{n}$ in corresponding domains in the complex plane in 
Theorem  \ref{thm:approxzn}.

\item \emph{Demonstrate advantage of cusped regions over smooth regions:} The polynomials we construct are bounded on regions with cusps and have a useful rapid growth property. We prove that polynomials that are bounded on an ellipse (or more generally in any region with a smooth boundary) cannot have this rapid growth property in Theorem \ref{ellipseslow}.

\item \emph{Develop novel static and dynamic momentum power methods:}
We generalize the static momentum power method of \cite{xu2018accelerated} and the dynamic momentum power methods of \cite{austin2024dynamically,cowal2025faber}  in  Algorithms \ref{alg:generalized-momentum} and  \ref{alg:dynamic-generalized-momentum}, which we demonstrate apply to a wider range of non-symmetric operators.   We prove theoretical convergence rates of these algorithms in Theorem
\ref{thm:general-momentum-convergence}, and conduct numerical experiments to demonstrate their advantages.
\end{enumerate}

\subsection{Preliminaries}
In what follows, we
construct families of polynomials that satisfy recurrence relations related to the expectation of an integer-valued mean-zero random variable, see, for example, \eqref{eq:chebyshev-expectation} above. As a consequence, iterating the recurrence relation is related to a random walk on the integers, which concentrates around zero. As the depth of the recurrence increases, the variance of the random walk increases, but as a trade-off, the region in the complex plane where these polynomials are bounded increases in area, asymptotically filling in the unit ball, e.g., see Figure \ref{fig:hypocycloids}.  
Our theoretical and numerical results highlight how this trade-off can be leveraged to develop efficient numerical methods for nonsymmetric problems. 

 We define the family of polynomials $P_n(z)$, $n \ge 0$ by $P_k(z) = z^k$ for $k \in \{0,1,\ldots,m-1\},$ and
\begin{equation} \label{recurrence}
P_{k+1}(z) = \frac{z - p_1}{p_0} P_k(z) - \sum_{j=2}^m \frac{p_j}{p_0} P_{k+1-j}(z), \quad \text{for} \quad k \ge m-1.
\end{equation}
 Here, $p = (p_0,\ldots,p_m) \in \mathbb{R}^{m+1}$ is a probability vector with non-negative entries $p_0,\ldots,p_m \ge 0$, that sum to $p_0 + \cdots + p_m = 1$, which we assume satisfies 
 \begin{equation} \label{meanzero}
 p_0 > 0 \quad \text{and} \quad
\sum_{j=0}^m (1-j) p_j = 0.
\end{equation}
Note that \eqref{recurrence} can be rearranged as
\begin{align}\label{eqn:rerecur}
\sum_{j=0}^{m} p_j P_{k+1-j}(z) = z P_k(z),
\end{align}
and  \eqref{meanzero} can be interpreted as a mean-zero condition for the random variable $Y$ satisfying
$\mathbb{P}(Y = 1-j) = p_j$ for $j \in \{0,\ldots,m\}$. By the definition of $Y$, we have
\begin{equation} \label{eq:expectationprop}
\mathbb{E} P_{k + Y}(z) = z P_k(z), \quad \text{for} \quad k \ge m-1.
\end{equation}
We also consider the variance of $Y$.  Since $\mathbb{E}Y = 0$, this is given by 
\begin{equation} \label{eq:variance}
\sigma^2 := \mathbb{E}Y^2 = \sum_{j = 0}^m (1 - j)^2 p_j.
\end{equation}
We later demonstrate that $\sigma$ governs the rate of rapid growth of $P_n(z)$ past $1$.

\subsubsection{Stability regions in the complex plane}

For each fixed $z \in \mathbb{C}$, the recurrence \eqref{recurrence}
is an $m$-th order difference equation with characteristic polynomial
\begin{equation} \label{eq:charpolyintro}
Q_z(r) = r^m + \frac{p_1-z}{p_0} r^{m-1} + \sum_{j=2}^m \frac{p_j}{p_0} r^{m-j},
\end{equation}
whose roots determine the growth rate of the recurrence, 
see, for example, \cite[Section 2.3]{Elaydi2005}. Solving $Q_z(r) =0$ for $z$ gives
\begin{equation} \label{eq:conformalmap1}
z = \sum_{j=0}^m p_j r^{1-j},
\end{equation}
which is a formula satisfied by the roots of $Q_z$. Hence, a necessary condition for the characteristic polynomial \eqref{eq:charpolyintro} to have a root of magnitude $1$ can be determined by setting $r = e^{i t}$ for $t \in [0,2\pi]$, which gives
\begin{equation} \label{eq:curveformula}
z(t) = \sum_{j=0}^m p_j e^{i(1-j) t},
\quad \text{for} \quad t \in [0,2\pi],
\end{equation}
where we write $z = z(t)$ to indicate the interpretation of $z$ as a function of $t$, see Figure \ref{fig:examplescurve}. In Theorem \ref{thm:boundedrapid}  we prove that the family of polynomials $P_n(z), n \ge 0$ is bounded in the region enclosed by \eqref{eq:curveformula}, as one might expect given how this region was constructed.  

\begin{remark}[Connection to Faber polynomials]
Our construction of random-walk-based polynomials starts with the recurrence \eqref{recurrence}, and then uses its  characteristic polynomial  \eqref{eq:charpolyintro} to determine the curve \eqref{eq:curveformula} that encloses the region where the polynomials are bounded. The Faber polynomials are constructed by the reverse process: given a region in the complex plane, the Faber polynomials are defined using a Riemann map from the exterior of the unit disk to the exterior of the given region. When the Riemann map is a rational function, the resulting Faber polynomials satisfy a recurrence relation, which can be determined from that map.

In the following, we define Faber polynomials to make this connection precise; see \cite{suetin1998series} for a detailed exposition of the Faber polynomials.  Let $\overline{\mathbb{C}}$  denote the extended complex plane, and $\Omega$ be a compact subset of the complex plane that is not a single point, whose complement $\overline{\mathbb{C}} \setminus \Omega$ is simply connected. By the Riemann mapping theorem, there is a unique conformal map $\psi : \{z \in \overline{\mathbb{C}}: |z|>1\} \to \overline{\mathbb{C}} \setminus \Omega$ such that $\psi(\infty) = \infty$ and $\psi'(\infty) >0$. The Faber polynomials $F_n(z)$, $n \ge 0$ associated  with the domain $\Omega$ are defined by the generating function
\begin{equation} \label{eq:generatefaber}
\frac{\psi'(r)}{\psi(r) - z} = \sum_{n=0}^\infty F_n(z) r^{-n-1} \quad \text{for} \quad |r| > 1 \quad \text{and} \quad z \in \Omega.
\end{equation}
To observe the connection  to the family of polynomials $P_n(z)$ defined above, the right-hand side of 
\eqref{eq:conformalmap1} defines a function
\begin{equation} \label{eq:conformalmap}
\psi(r) = \sum_{j=0}^m p_j r^{1-j},
\end{equation}
which is a conformal map from the exterior of the unit disc $\{z \in \overline{\mathbb{C}} : |z| > 1\}$ onto its image  $\psi(\{z \in \overline{\mathbb{C}} : |z| > 1\}) = \overline{\mathbb{C}} \setminus \Gamma$, where $\Gamma$ is the closure of the region enclosed by the curve  \eqref{eq:curveformula}. Note that $\psi(\infty) = \infty$ and $\psi'(\infty) > 0$.  Hence, the Faber polynomials for $\Gamma$ can be generated by substituting  \eqref{eq:conformalmap} into \eqref{eq:generatefaber}. Using \eqref{eq:generatefaber}, one can show that the Faber polynomials on the region enclosed by
$$
z(t) = \sum_{j=0}^m p_j e^{i(1-j) t},
\quad \text{for} \quad t \in [0,2\pi],
$$
satisfy the recurrence \eqref{recurrence}. 
This can be observed by multiplying both sides of \eqref{eq:generatefaber} by 
$\psi(r) -z$,
and applying the definition of \eqref{eq:conformalmap}. Since the largest negative power of $\psi'(r)$ is $r^{-m}$, the coefficients of $r^{-n-1}$ must vanish for $n \ge m$, which corresponds to $F_n(z)$ satisfying the same recurrence  \eqref{eqn:rerecur} that $P_n(z)$ satisfies; also see  \cite[Chapter 2]{suetin1998series}.
Hence, the polynomials we consider are essentially an instance of Faber polynomials (the only difference is that we consider slightly different initial conditions adapted to our purposes).  Finally, to avoid overstating or mischaracterizing the link to Faber polynomials, we note that many polynomials defined by a recurrence can, in a sense, be viewed as Faber polynomials (up to the choice of initial conditions) for some domain in the complex plane defined by a rational Riemann map. 
We make use of the corresponding
domain in the complex plane both in our theoretical results and computational methods.
\end{remark}

\begin{figure}[h!]
\centering
\begin{tikzpicture}[scale=1.5]
\draw (-1.25, 0) -- (1.25, 0);
\draw (0, -1.25) -- (0, 1.25);
\draw (1, 0) node[anchor=south west]{$1$};
\draw (0, 1) node[anchor=south west]{$i$};
\draw [dashed] (0,0) circle (1);
\draw [ultra thick, domain = 0:360, samples = 60, line join=round]
    plot[smooth] ({7/12*cos(\x) + 3/12*cos(-\x) + 2/12*cos(-2*\x)}, {7/12*sin(\x) + 3/12*sin(-\x) + 2/12*sin(-2*\x)});
\end{tikzpicture}
\quad
\begin{tikzpicture}[scale=1.5]
\draw (-1.25, 0) -- (1.25, 0);
\draw (0, -1.25) -- (0, 1.25);
\draw (1, 0) node[anchor=south west]{$1$};
\draw (0, 1) node[anchor=south west]{$i$};
\draw [dashed] (0,0) circle (1);
\draw [ultra thick, domain = 0:360, samples = 60, line join=round]
    plot[smooth] ({2/3*cos(\x) + 1/3*cos(-2*\x)}, {2/3*sin(\x) + 1/3*sin(-2*\x)});
\end{tikzpicture}
\quad
\begin{tikzpicture}[scale=1.5]
\draw (-1.25, 0) -- (1.25, 0);
\draw (0, -1.25) -- (0, 1.25);
\draw (1, 0) node[anchor=south west]{$1$};
\draw (0, 1) node[anchor=south west]{$i$};
\draw [dashed] (0,0) circle (1);
\draw [ultra thick, domain = 0:360, samples = 60, line join=round]
    plot[smooth] ({5/8*cos(\x) + 2/8*cos(-\x) + 1/8*cos(-3*\x)}, {5/8*sin(\x) + 2/8*sin(-\x) + 1/8*sin(-3*\x)});
\end{tikzpicture}
    \caption{The curve \eqref{eq:curveformula} for $p = (\frac{7}{12},0,\frac{3}{12},\frac{2}{12})$, $(\frac{2}{3},0,0,\frac{1}{3})$, and $(\frac{5}{8},0,\frac{2}{8},0,\frac{1}{8})$ (ordered left to right). 
    }
    
    \label{fig:examplescurve}
\end{figure}
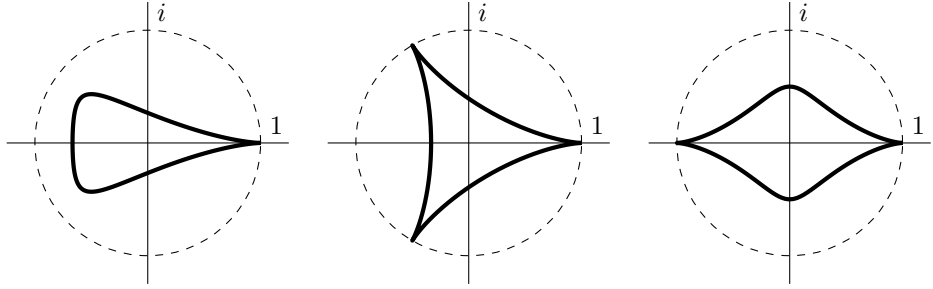

\subsubsection{Hypocycloids} \label{sec:hypocycloid}
In the special case where $p_0$ and $p_m$ are the only non-zero entries of $p = (p_0,\ldots,p_m)$, for $m \ge 2$, the condition \eqref{meanzero} implies
$$
p_0 = \frac{m-1}{m} \quad \text{and} \quad p_m = \frac{1}{m}.
$$
We denote the family of polynomials corresponding to these probabilities as $P_n^{(m)}(z)$, $n \ge 0$. 
For 
these polynomials, 
recurrence \eqref{recurrence} 
takes the form
\begin{equation}\label{eqn:genrec}
P^{(m)}_{n+1}(z) = \frac{m}{m-1}z P_{n}^{(m)}(z) - \frac{1}{m-1} P^{(m)}_{n-m+1}(z), \quad \text{for} \quad n \ge m-1.
\end{equation}
Furthermore, in this case, the curve \eqref{eq:curveformula} has the form
\begin{equation} \label{eq:gamma-higher-order}
z^{(m)}(t) = \frac{m-1}{m}e^{it} + \frac{1}{m}e^{(1-m)it} \quad \text{for}\ t \in [0, 2\pi],
\end{equation}
which are hypocycloid regions in the complex plane, see Figure \ref{fig:hypocycloids}. In the following, we refer to $P_n^{(m)}$ as the hypocycloid polynomials.

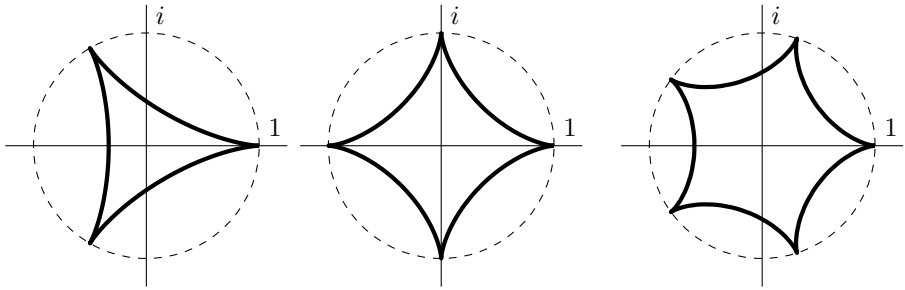
\begin{figure}[h!]
\centering
\begin{tikzpicture}[scale=1.5]
\draw (-1.25, 0) -- (1.25, 0);
\draw (0, -1.25) -- (0, 1.25);
\draw (1, 0) node[anchor=south west]{$1$};
\draw (0, 1) node[anchor=south west]{$i$};
\draw [dashed] (0,0) circle (1);
\draw [ultra thick, domain = 0:360, samples = 60, line join=round]
    plot[smooth] ({2/3*cos(\x) + 1/3*cos(-2*\x)}, {2/3*sin(\x) + 1/3*sin(-2*\x)});
\end{tikzpicture}
\quad
\begin{tikzpicture}[scale=1.5]
\draw (-1.25, 0) -- (1.25, 0);
\draw (0, -1.25) -- (0, 1.25);
\draw (1, 0) node[anchor=south west]{$1$};
\draw (0, 1) node[anchor=south west]{$i$};
\draw [dashed] (0,0) circle (1);
\draw [ultra thick, domain = 0:360, samples = 60, line join=round]
    plot[smooth] ({3/4*cos(\x) + 1/4*cos(-3*\x)}, {3/4*sin(\x) + 1/4*sin(-3*\x)});
\end{tikzpicture}
\quad
\begin{tikzpicture}[scale=1.5]
\draw (-1.25, 0) -- (1.25, 0);
\draw (0, -1.25) -- (0, 1.25);
\draw (1, 0) node[anchor=south west]{$1$};
\draw (0, 1) node[anchor=south west]{$i$};
\draw [dashed] (0,0) circle (1);
\draw [ultra thick, domain = 0:360, samples = 60, line join=round]
    plot[smooth] ({4/5*cos(\x) + 1/5*cos(-4*\x)}, {4/5*sin(\x) + 1/5*sin(-4*\x)});
\end{tikzpicture}
    \caption{The curve \eqref{eq:curveformula} for $p = (\frac{2}{3},0,\frac{1}{3})$, $(\frac{3}{4},0,0,\frac{1}{4})$, and $(\frac{4}{5},0,0,0,\frac{1}{5})$ (ordered left to right). These curves are called a 
    $3$-cusped hypocycloid (deltoid), a $4$-cusped hypocycloid (astroid), and a $5$-cusped hypocycloid, respectively.}
    \label{fig:hypocycloids}
\end{figure}

\begin{remark}\label{rem:p1zero}
In the remainder, we will make the assumption that $p_1 = 0$.  As can be observed from \eqref{recurrence}, this parameter serves essentially as a ``shift,'' and its use in algorithms for applications such as the interior eigenvalue problem \cite{FaSa12} may merit separate investigation. 
Where it simplifies the presentation, $p_1$ may appear in summations such as \eqref{eqn:rerecur} and \eqref{eq:variance}, and will be understood to be zero.
\end{remark}

\subsection{Main analytic results}

Our first main result characterizes 
the region enclosed by the 
curve  $z(t)$ 
defined in \eqref{eq:curveformula}. We refer to the region enclosed by this curve as the stability region for the family of polynomials $P_n(z)$, $n \ge 0$.

Recall that $z(t)$ and the family of polynomials $P_n(z)$ are defined using a probability vector $p = (p_0,\ldots,p_m) \in \mathbb{R}^{m+1}$, which satisfies \eqref{meanzero}, which states $\sum_{j=0}^m (1-j) p_j = 0$.
When $p_0 = 1/2$ and $p_2 = 1/2$, the family of polynomials $P_n(z)$ is the family of Chebyshev polynomials of the first kind $T_n$, see Section 
\ref{subsec:bg}. In this case, the curve $z(t)$ defined in \eqref{eq:curveformula}, is given by
$$
z(t) = \frac{1}{2} e^{i t} + \frac{1}{2} e^{- i t} = \cos(t),
\quad \text{for} \quad t \in [0,2\pi],
$$
whose range is $z([0,2\pi]) = [-1,1]$, which is the set in $\mathbb{C}$ where the Chebyshev polynomials $T_n$ are bounded. The following result characterizes 
the stability region under the assumption $p_j > 0$ for some $j \ge 3$, which avoids the degenerate Chebyshev case.

\begin{theorem}[Description of stability region] 
\label{thm:charcurve}
Let $\gamma$ be the closed curve parameterized by $z(t)$ for $t \in [0,2\pi]$ defined in \eqref{eq:curveformula}, for a probability vector $p = (p_0,\ldots,p_m)$ satisfying \eqref{meanzero} with $p_1 = 0$ and $p_j > 0$ for some $j \ge 3$. Then, $\gamma$ is continuously differentiable except at $n$ cusps located at $e^{2\pi i k/n}$ for $k \in \{0,\ldots,n-1\}$, where 
$$
n =\gcd(\{j \in \{2,\ldots,m\} : p_j > 0\}).
$$
Moreover, the region enclosed by $\gamma$ is a radially convex set
with non-empty interior.
\end{theorem}

The proof of Theorem \ref{thm:charcurve} is given in Section \ref{sec:proofofthmcharcurve}. The fact that the cusps are characterized by the greatest common divisor of the indices $j$ of the non-zero probabilities for $j \ge 2$ can be seen in Figure \ref{fig:examplescurve}; also see Figure \ref{figwherearecusps} below. 

Our second main result shows that $P_n(z)$ is bounded on the region enclosed by the curve $\gamma$ and has a rapid growth property past $1$ in the sense of Definition \ref{defn:rapidgrowth}.

\begin{theorem}[Boundedness and rapid growth] \label{thm:boundedrapid}
Let $P_n(z)$, $n \ge 0$ be the family of polynomials defined by \eqref{recurrence} for a probability vector $p = (p_0,\ldots,p_m)$ that satisfies \eqref{meanzero} with $p_1 = 0$. Let $\Gamma$ denote the closure of the region enclosed by the curve $\gamma$ parameterized by \eqref{eq:curveformula}. Then,
$$
|P_n(z)| \le C \quad \text{for all} \quad z \in \Gamma \quad \text{and} \quad n \ge 0,
$$
and
$$
|P_n(1 + \varepsilon)| \ge c (1 + \sigma^{-1} \sqrt{2\varepsilon})^n, \quad \text{for all} \quad \varepsilon > 0 \quad
\text{and} \quad n \ge 0,
$$
where $C,c > 0$ are  constants that only depend on $p$, and 
$\sigma$ is defined in \eqref{eq:variance}.
\end{theorem}

The proof of Theorem \ref{thm:boundedrapid} is given in Section \ref{sec:proofthmboundedrapid}.
The asymptotic growth rate of $P_n(1 + \varepsilon)$ is proportional to the inverse of the standard deviation $\sigma^{-1}$ of the random variable $Y$ determining the random walk: the larger the standard deviation, the slower $P_n$ grows. As a trade-off, the region where $P_n$ is bounded gets larger, see Figure \ref{fig:hypocycloids}. We numerically illustrate the growth rates established in Theorem \ref{thm:boundedrapid} for the special case of $m$-cusped hypocycloid recurrence relations
\eqref{eqn:genrec}, see Figure \ref{fig:hypocycloid-growth}.

\begin{figure}[h!]
    \centering
    \includegraphics[width=0.7\linewidth]{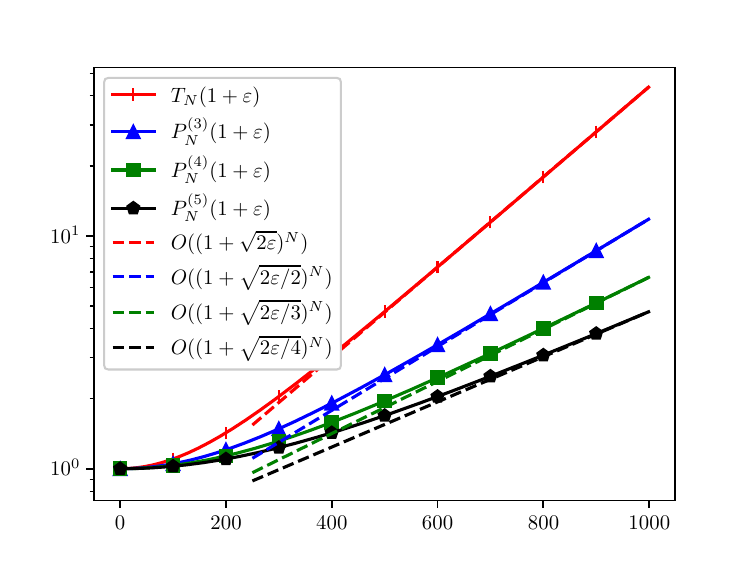}
    \caption{$P_n^{(m)}(1+\varepsilon)$ for $\varepsilon = 10^{-5}$ for $n\in\{1,\ldots,1000\}$ and $m \in \{2,3,4,5\}$, along with the rates predicted by Theorem \ref{thm:boundedrapid},   where $P_n^{(m)}$ is the hypocycloid polynomial defined in \eqref{eqn:genrec} and associated with the hypocycloid curves illustrated in Figure \ref{fig:hypocycloids}.}
    \label{fig:hypocycloid-growth}
\end{figure}

Our third main result says that $z^n$ is approximately a polynomial of degree $\sim \sqrt{n}$ in the region enclosed by \eqref{eq:curveformula}.

\begin{theorem}[Approximation of $z^n$ by a polynomial of degree $\sim \sqrt{n}$] \label{thm:approxzn}
Let $\Gamma$ be the closure of the region enclosed by the curve $\gamma$ parameterized by \eqref{eq:curveformula}, and $P_n(z)$, $n \ge 0$, be the polynomials defined by \eqref{recurrence} for a probability vector $p = (p_0,\ldots,p_m)$ that satisfies  \eqref{meanzero} with $p_1 = 0$. Then, 
$$
\left| \sum_{k = 0}^{\lfloor t \sqrt{n} \rfloor} \alpha_k P_k(z) - z^n \right| \le C e^{-c t^2},
\quad \text{for all} \quad z \in \Gamma,  \quad n\ge 0 \quad \text{and} \quad t > 0,
$$
where $C,c > 0$ are constants that only depend on $p$, the coefficients $\alpha_k = \mathbb{P}(|Y_n| = k)$ are probabilities, and $Y_n$ is defined in Section \ref{sec:proofthmapproxzn} below.
\end{theorem}

The proof of Theorem \ref{thm:approxzn} is given in Section \ref{sec:proofthmapproxzn}.
Our final main result provides insight into the connection between the cusps in Theorem \ref{thm:charcurve}  and rapid growth in Theorem \ref{thm:boundedrapid}. The cusps present in  Figures \ref{fig:examplescurve} and \ref{fig:hypocycloids} restrict the applications of these polynomials (see Section \ref{sec:applications} for details). 
A natural question arises: is it still possible for a polynomial $P_n$ to achieve rapid growth (in the sense of Definition \ref{defn:rapidgrowth})
while being bounded on a region without a cusp?  The following arguments demonstrate that cusps are necessary for fast growth.  
Essentially, if $P_n(z)$ is a polynomial of degree $n$ that is bounded on an ellipse with vertices $\pm 1$ and co-vertices $\pm i \delta$ for $1 > \delta > 0$, then 
$$
|P_n(1+\varepsilon)| \le c_1 (1 + c_1' \varepsilon)^n,
$$
for some constants $c_1,c'_1 > 0$. In contrast, the polynomials 
presented in \eqref{recurrence} satisfy
$$
|P_n(1+\varepsilon)| \ge c_2 (1+ c_2' \sqrt{\varepsilon})^n,
$$
for some constant $c_2,c_2' > 0$, see Theorem \ref{thm:boundedrapid}.
As a consequence, 
the higher-order momentum methods we present in Section \ref{sec:applications}, based on this polynomial construction which corresponds to a higher-order recurrence, 
can lead to arbitrarily faster rates when $\varepsilon > 0$ is small.
In linear algebra applications, the case where $\varepsilon > 0$ is small corresponds to a small spectral gap, which is a case of practical interest.

\begin{theorem}[Bound on growth of polynomial bounded on an ellipse]
\label{ellipseslow}
Fix $1 > \delta > 0$ and let
$\Omega_\delta$ be the closure of the region enclosed by the ellipse in $\mathbb{C}$ with vertices $\pm 1$ and co-vertices $\pm i \delta$. Suppose that $P_n$ is a polynomial of degree at most $n$ such that $|P_n(z)| \le 1$ for $z \in \Omega_\delta$. Then,
$$
|P_n(1+\varepsilon)| \le \left( 1 + \frac{3}{2\delta} \varepsilon \right)^n,
$$
for all $1 > \varepsilon > 0$.
\end{theorem}

The proof of Theorem \ref{ellipseslow} is given in Section \ref{proofellipseslow}. An immediate consequence of Theorem \ref{ellipseslow} is that if $P_n$ is a polynomial bounded on a set $\Omega$ that contains an ellipse, which is tangent to the boundary at $1$, then rapid growth in the sense of Definition \ref{defn:rapidgrowth} is not possible.
Put differently, rapid growth past $1$ for a polynomial bounded on a domain that contains $1$ on the boundary is only possible if there is no ellipse that intersects $1$ and is contained in the domain, such as when there is a cusp, see Figure \ref{cuspcorner}. Note the possibility of scaling, translating, and rotating the ellipse in Theorem \ref{ellipseslow}. This result can be used to say that rapid growth is never possible at a point on the boundary when an ellipse enclosed in the domain can intersect the point (which happens whenever the boundary is smooth).

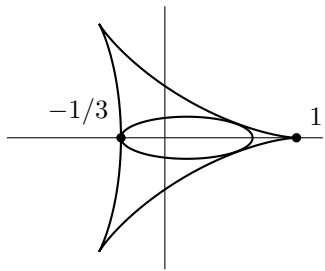
\begin{figure}[h!]
\begin{tikzpicture}[scale=1.75]
\draw (-1.2, 0) -- (1.2, 0);
\draw (0, -1) -- (0, 1);

\draw (1,0) node[circle,fill,inner sep=1.25pt,label=north east:{$1$}] {};
\draw (-1/3,0) node[circle,fill,inner sep=1.25pt,label=north west:{$-1/3$}] {};

\draw [thick, domain = 0:360, samples = 60]
    plot[smooth] ({2/3*cos(\x) + 1/3*cos(-2*\x)}, {2/3*sin(\x) + 1/3*sin(-2*\x)});
    
\draw [thick,domain = 0:360, samples = 60]
    plot[smooth] ({1/2*cos(\x)+1/6}, {.16*sin(\x)});
\end{tikzpicture}
\caption{As a consequence of Theorem \ref{ellipseslow}, if a polynomial is bounded on the deltoid region, then rapid growth in the sense of \eqref{eq:defrapidgrowth} is possible at $1$ but impossible at $-1/3$, since there is an ellipse contained in the deltoid region tangent to this point.}\label{cuspcorner}
\end{figure}

\section{Accelerating power iteration}
\label{sec:applications}
In this section, 
we apply the polynomial theory developed in the previous section to generalize the algorithm presented in \cite{austin2024dynamically}.
This allows the acceleration of the power method for broader classes of nonsymmetric problems. For a generalized momentum power iteration to be successful, the complex eigenvalues of a matrix $A$ need to be contained within the stability region for $P_k$.  The flexibility of defining the polynomial $P_k$ increases the applicability of the method. 

First we present the convergence and acceleration properties of Algorithm \ref{alg:generalized-momentum}, which uses fixed momentum parameters.  However, the optimal parameter choice depends on knowledge of $\lambda_2$. In Subsection \ref{subsec:dynamic} we present Algorithm \ref{alg:dynamic-generalized-momentum}, which dynamically approximates the optimal parameters without this knowledge.

\subsection{Notation and assumptions} \label{notationassumptions} Let $\vect{A} \in \mathbb{C}^{n \times n}$ be a diagonalizable matrix. 
Let $\lambda_1,\ldots,\lambda_n$ denote the eigenvalues of $\vect{A}$. Assume that the eigenvalues are ordered in order of descending magnitude and that 
$$
\lambda_1 > |\lambda_2| \ge \cdots \ge |\lambda_n|.
$$
Let $\vec{\phi}_1,\ldots,\vec{\phi}_n$ denote the corresponding eigenvectors, which we assume are normalized such that $\|\vec{\phi}_j\|_2 = 1$. Note that we assume that $\lambda_1$ is real and positive.

\subsection{The Generalized Momentum Power Method} The core of the generalized momentum power method is the recurrence
\begin{equation}\label{eq:general-momentum-recurrence}
\vec{w}_{k+1} = \vect{A}\vect{w}_k - \beta_1\vec{w}_{k-1} - \ldots - \beta_{m-1} \vec{w}_{k - m + 1},
\end{equation}
where $\beta_1, \dots, \beta_{m-1}$ are called the momentum parameters. 
To compute the initial conditions for the generalized momentum power method, we compute $m-1$ iterations of the power method (Algorithm \ref{alg:power-method}) applied to the rescaled matrix $p_0\vect{A}$ for some parameter $0 < p_0 \leq 1$.

\begin{remark}\label{rem:beta0}
A more general form of \eqref{eq:general-momentum-recurrence} is 
$\vec{w}_{k+1} = (\vect{A} - \beta_0 \vec{I})\vect{w}_k - \beta_1\vec{w}_{k-1} - \ldots - \beta_{m-1} \vec{w}_{k - m + 1},$
where the shift parameter $\beta_0$ corresponds to probability $p_1$, assumed zero herein as per Remark \ref{rem:p1zero}.
\end{remark}

To fix notation, we next state the power method, followed by the static version of the generalized momentum method.  We refer to this method as static, as the momentum parameters $\beta_1, \ldots \beta_{m-1}$ remain fixed throughout the algorithm.
\begin{algorithm}[h!]
\caption{Power method}
\label{alg:power-method}
\begin{algorithmic}[1]
\REQUIRE Matrix $\vect{A} \in \mathbb{R}^{n \times n}$, vector $\vect{v}_0 \in \mathbb{R}^n$,  $N \in \mathbb{N}$
\STATE Set $h_0 = \|\vect{v}_0\|_2$ and $\vect{x}_0 = h_{0}^{-1}\vect{v}_{0}$
\FOR{$k = 0,\dots,N-1$}
    \STATE Set $\vect{v}_{k+1} = \vect{A} \vect{x}_k$
    \STATE Set $h_{k+1} = \|\vect{v}_{k+1}\|_2$ and $\vect{x}_{k+1} = h_{k+1}^{-1}\vect{v}_{k+1}$
    \STATE Set $\nu_{k} = \langle \vect{v}_{k+1}, \vec{x}_k \rangle$ and $d_{k} = \| \vect{v}_{k+1} - \nu_{k}\vect{x}_k\|_2$ \hfill $\triangleright$ \emph{defined for later use}
\ENDFOR
\RETURN $\vec{x}_N$
\end{algorithmic}
\end{algorithm}

\begin{algorithm}
\caption{Generalized Momentum Power Method}
\label{alg:generalized-momentum}
\begin{algorithmic}[1]
\REQUIRE  Matrix $\vect{A} \in \mathbb{R}^{n \times n}$, vector $\vect{v}_0  \in \mathbb{R}^n$, iteration count $N \in \mathbb{N}$,  probability vector $p = (p_0,\ldots,p_m)$ satisfying \eqref{meanzero} with $p_1 = 0$, and $0 <\lambda_\ast < \lambda_1$. 
\STATE Do $m-1$ iterations of power iteration with inputs $p_0\vect{A}$ and $\vec{v}_0$
\STATE 
Compute the momentum parameters 
$\beta_j = p_{j+1}p_0^{j}\lambda_*^{j+1}$, for  $j = 1, \dots, m-1$
\FOR{{$k = m-1,\ldots,N-1$}}
\STATE{Set $\vect{v}_{k+1} = \vect{A} \vect{x}_{k}$}
\STATE{Set $\vect{u}_{k+1,0}' = \vect{v}_{k+1}$}
\STATE{Set $H_{k, 0} = 1$}
\FOR{ $j = 1, \dots, m-1$}
\STATE{Set $H_{k, j} = H_{k, j-1} h_{k+1-j}$}
\STATE{Set $\vect{u}'_{k+1, j} = \vect{u}'_{k+1,j-1} -  (\beta_j/H_{k, j}) \vect{x}_{k-j}$}
\ENDFOR
\STATE{Set $\vect{u}_{k+1} = \vect{u}'_{k+1, m-1}$}
\STATE{Set $h_{k+1} = \|\vect{u}_{k+1}\|_2$ and $\vect{x}_{k+1} = h_{k+1}^{-1} \vect{u}_{k+1}$}
\ENDFOR
\RETURN $\vec{x}_N$
\end{algorithmic}
\end{algorithm}

The connection between Algorithm \ref{alg:generalized-momentum} and the polynomial construction \eqref{recurrence} is made precise in the next lemma.
\begin{lemma}
\label{lem:general-momentum-works}
Let $p = (p_0,\ldots,p_m)$ be a probability distribution satisfying \eqref{meanzero} with $p_1 = 0$, and let $P_k(z)$ denote the polynomial family defined by  
\eqref{recurrence}.
Then running Algorithm \ref{alg:generalized-momentum} with parameters $\vect{A}$, $\vec{v}_0$, $N$, $p$, 
and $\lambda_\ast$, computes  
$\vec{x}_N = \vec{w}_N/\|\vec{w}_N\|_2$ where
$\vec{w}_N = P_N(\vect{A}/\lambda_*)\vec{v}_0$.
\end{lemma}

\begin{proof}[Proof of Lemma \ref{lem:general-momentum-works}]
The proof strategy for this lemma is based on reframing the recurrence defining $P_k(z)$ into a form that directly corresponds to the presentation given in Algorithm \ref{alg:generalized-momentum}.  
To translate the growth properties of the polynomials $P_k$ into a generalized momentum-based power method, we first dilate and rescale $P_k$ to a family of monic polynomials $\tilde P_k$, such that 
$\tilde P_k(z) = C^kP_k(z/\lambda_*),$ for some constant $C$.
To determine $C$, we substitute $z/\lambda_*$ into the recurrence \eqref{recurrence} and find
$$ P_{k+1}(z/\lambda_*) = \frac{z}{\lambda_*p_0} P_k(z/\lambda_*) - \sum_{j=2}^m \frac{p_j}{p_0} P_{k+1-j}(z/\lambda_*). $$
Replacing $P_k(z/\lambda_*)$ with $C^{-k}\tilde P_k(z)$ gives us the recurrence
$$ C^{-(k+1)}\tilde P_{k+1}(z) = \frac{z}{\lambda_*p_0} C^{-k} \tilde P_k(z)  - \sum_{j=2}^m \frac{p_j}{p_0} C^{-k-1+j} \tilde P_{k+1-j}(z),$$
and multiplying through by $C^{k+1}$ gives us
$$ \tilde P_{k+1}(z) = \frac{z}{\lambda_*p_0} C \tilde P_k(z)  - \sum_{j=2}^m \frac{p_j}{p_0} C^{j} \tilde P_{k+1-j}(z).$$
Setting $C = \lambda_* p_0$ ensures that $\tilde P_k$ is monic.  Moreover, we have determined a recurrence relation for $\tilde P_k$ for $k \geq m - 1$, given below:
\begin{equation} \label{eq:monic-recurrence}
    \tilde P_{k+1}(z) = z \tilde P_k(z) - \sum_{j=2}^m p_j p_0^{j-1} \lambda_*^{j} \tilde P_{k+1-j}(z)
   = z\tilde P_k(z) - \sum_{j=2}^m \beta_{j-1} \tilde P_{k + 1 - j}(z),
\end{equation}
where $\beta_{j} = p_{j+1}p_0^{j}\lambda_*^{j+1}$ are the momentum parameters.

For the $m-1$ initial iterations corresponding to computing $P_0(z) = 1$ through $P_{m-1}(z) = z^{m-1}$, we observe that $$\tilde{P}_k(z) = p_0^kz^k \text{ for } k \in \{0, 1, \dots, m-1\}.$$
This corresponds to applying $m-1$ iterations of the power iteration to the rescaled operator $p_0\vect{A}$ rather than $\vect{A}$.
 
To ensure that our algorithm does not result in unbounded growth,  
we define a consistent normalization as follows.
Define $\vec{x}_k = \vect{w}_k/\|\vect{w}_k\|_2$ for $k \geq 0$ and $h_k = \|\vect{w}_k\|_2/\|\vect{w}_{k-1}\|_2$ for $k \geq 1$.  Also define $h_0 = \|\vec{w}_0\|_2$.  It follows that 
$$\vect{w}_k = \left(\prod_{i = 0}^k h_i\right)\vec{x}_k.$$
We can use this to rewrite equation \eqref{eq:general-momentum-recurrence} in terms of $\vec{x}_k$:
$$\left(\prod_{i = 0}^{k+1} h_i\right)\vec{x}_{k+1} = \left(\prod_{i = 0}^{k} h_i\right)\vect{A}\vec{x}_k - \sum_{j=1}^{m-1} \left(\prod_{i = 0}^{k - j} h_i\right)\beta_j\vec{x}_{k-j}.$$
Dividing both sides by $\left(\prod_{i = 0}^{k} h_i \right)$, we arrive at the following recurrence:
$$h_{k+1}\vec{x}_{k+1} = \vect{A}\vec{x}_k -\sum_{j=1}^{m-1} \left(\prod_{i = k - j + 1}^{k} h_i\right)^{-1}\beta_j\vec{x}_{k-j} = \vec{A}\vec{x}_k - \sum_{j=1}^{m-1}H_{k, j}^{-1}\beta_j\vect{x}_{k-j},$$
where we denote
$H_{k, j} = \prod_{i = k - j + 1}^k h_i.$ Finally, we observe that each iteration of Algorithm \ref{alg:generalized-momentum} applies the recurrence \eqref{eq:monic-recurrence} with the parameters $\beta_j$ defined above and $\vect{A}$ in place of $z$, along with a normalization step.  We therefore conclude that the output of Algorithm \ref{alg:generalized-momentum} will satisfy $\vect{x}_N = \vect{w}_N/\|\vect{w}_N\|_2$, with $\vect{w}_N = P_N(\vec{A}/\lambda_*)\vect{v}_0$.
\end{proof}

The next theorem quantifies the convergence rate of Algorithm \ref{alg:generalized-momentum}.
For a given spectrum, the convergence rate of Algorithm \ref{alg:generalized-momentum} is governed by the growth of $P_N(\lambda_1/\lambda_*)$.
The growth rate established in Theorem \ref{thm:boundedrapid} is sharp up to lower order terms (see the proof of Lemma \ref{thm:root-estimate}). 
\begin{theorem}\label{thm:general-momentum-convergence} Let $p = (p_0, \dots, p_m) \in \mathbb{R}^{m+1}$ be a probability vector such that \eqref{meanzero} is satisfied with $p_1 = 0$.  Let $\Gamma$ be the closure of the region enclosed by the curve parameterized by \eqref{eq:curveformula}.  In addition to the notation and assumptions of \S \ref{notationassumptions}, 
assume that $\lambda_*$ is given such that $0 < \lambda_* < \lambda_1$, and 
\begin{align}\label{eqn:sd-domain}
\lambda_2, \dots, \lambda_n \in  \lambda_* \Gamma = \{ \lambda_* z \in \mathbb{C} : z \in \Gamma\}.
\end{align}
Let $\vec{v}_0 \in \mathbb{C}^n$ be an initial vector with eigenbasis expansion 
$\vec{v}_0 = \sum_{j=1}^n a_j \vec{\phi}_j$, with $a_1 \ne 0$. Run Algorithm \ref{alg:generalized-momentum} with inputs $\vec{A}$, $\vec{v}_0$, 
$p$, $N$, and $\lambda_*$. Then the output $\vec{x}_N$ satisfies
\begin{align}\label{eqn:thm21}
\vec{x}_N = 
\frac{a_1}{|a_1|}
\frac{\vec{\phi}_1 + R_N}{\| \vec{\phi}_1 + R_N\|_2},
~\text{ with } \|{R_N}\|_2 \le \widehat C 
\left(1 + \sigma^{-1} \sqrt{2 \varepsilon_\ast}\right)^{-N},
\end{align}
 where  $\varepsilon_\ast = |{\lambda_1}/{\lambda_*}| - 1$, and 
$\widehat C =c^{-1}C\sum_{j = 2}^n |a_j/a_1|$, with $C$ and $c$ the respective boundedness and rapid growth constants from Theorem \ref{thm:boundedrapid}, which depend only on $p$.
The ratios $|a_j/a_1|$ depend on the initial iterate $\vect{v}_0$.
\end{theorem}

The proof of Theorem \ref{thm:general-momentum-convergence}  is given in Section \ref{sec:proofthm:general-momentum-convergence}.
Additionally, Theorem \ref{thm:general-momentum-convergence} immediately leads to a corollary with practical importance.

\begin{corollary}\label{cor:l2opt}
A necessary condition for \eqref{eqn:sd-domain} to hold is $\lambda_\ast \ge |\lambda_2|$. If \eqref{eqn:sd-domain} holds when $\lambda_\ast = |\lambda_2|$, then this choice of $\lambda_*$ minimizes the  convergence bound in \eqref{eqn:thm21} over all possible values of $\lambda_*$.
\end{corollary}
Corollary \ref{cor:l2opt},
which follows from the observation that $\varepsilon_\ast$ is maximized when $\lambda_\ast = |\lambda_2|$,
implies that achieving fast convergence requires either a priori knowledge of $\lambda_2$ (which is generally unrealistic) or an efficient strategy to estimate $\lambda_2$. We present such an approximation strategy in Subsection \ref{subsec:dynamic}, together with a dynamic implementation of Algorithm \ref{alg:generalized-momentum}, given as Algorithm \ref{alg:dynamic-generalized-momentum}. In light of \eqref{eqn:sd-domain}, we will restrict our attention to problems where $\lambda_1 > \lambda_2 > 0$ are real. The numerical results in Subsection \ref{sec:numerics} demonstrate the efficiency of this approach.

\begin{remark}[Example: Hypocycloids]

Set $p_0 = (m-1)/m$ and $p_m = 1/m$, with $p_k = 0$ for $1 \leq k \leq m-1$.
Then \eqref{eq:variance} gives us
$$\sigma^2 = \frac{m-1}{m} \cdot 1^2 + \frac{1}{m} \cdot (m-1)^2 = m - 1,$$
which 
from Theorem \ref{thm:root-estimate}
corresponds to a dominant zero of the characteristic polynomial
$$r = 1 + (m - 1)^{-1/2}\sqrt{2\varepsilon} = 1 + \sqrt{\frac{2}{m - 1}\varepsilon}.$$
This matches the existing rates that we've established.  For the Chebyshev polynomials, $m = 2$, which corresponds to the established growth rate of $(1 + \sqrt{2\varepsilon})^n$, and for the Faber polynomials on the deltoid, $m = 3$ gives a rate of $(1 + \sqrt{\varepsilon})^n$.
The momentum parameters given by Lemma \ref{lem:general-momentum-works} corresponding to three-term recurrences defined by $p_0 = (m-1)/m$ and $p_m = 1/m$ are shown in Table \ref{tab:three-term-parameters}.
\begin{table}[h!]
    \centering
    \begin{tabular}{c|l|c|l}
        Order & Iteration & Optimal $\beta$ & Error\\
        \hline
        $2$ & $\vec{x}_{k+1} = \vect{A}\vec{x}_{k} - \beta \vec{x}_{k - 1}$ & $(1/4)\lambda_2^2$ & $\mathcal{O}\big((1 + \sqrt{2\varepsilon})^{-N}\big)$\\[5pt]
        $3$ & $\vec{x}_{k+1} = \vect{A}\vec{x}_{k} - \beta \vec{x}_{k - 2}$ & $(4/27)\lambda_2^3$ & $\mathcal{O}\big((1 + \sqrt{\varepsilon})^{-N}\big)$\\[5pt]
        $4$ & $\vec{x}_{k+1} = \vect{A}\vec{x}_{k} - \beta \vec{x}_{k - 3}$ & $(27/256)\lambda_2^4$ & $\mathcal{O}\big((1 + \sqrt{2\varepsilon/3})^{-N}\big)$\\[5pt]
        $m$ & $\vec{x}_{k+1} = \vect{A}\vec{x}_{k} - \beta \vec{x}_{k + 1 - m}$ & $\dfrac{(m - 1)^{m - 1}}{m^{m}} \lambda_2^m$ & $\mathcal{O}\left(\left(1 + \sqrt{\dfrac{2\varepsilon}{m-1}}\right)^{-N}\right)$
    \end{tabular}
    
    \vspace{2pt}
    \caption{The optimal momentum coefficients for accelerating the power method via three-term recurrences, as well as the expected convergence rates.}
    \label{tab:three-term-parameters}
\end{table}

\end{remark}

\subsection{The dynamic generalized momentum power method}\label{subsec:dynamic}
As shown in Theorem \ref{thm:general-momentum-convergence} and Corollary \ref{cor:l2opt}, given a probability vector $p = (p_0,\ldots,p_m)$ satisfying \eqref{meanzero} with $p_1 = 0$, the momentum parameters for Algorithm \ref{alg:generalized-momentum} given by $\beta_j = p_{j+1}p_0^{j}\lambda_*^{j+1}$ optimize the performance of the algorithm for $\lambda_\ast = \lambda_2$ so long as \eqref{eqn:sd-domain} holds. Our goal in this section is to develop a strategy that efficiently approximates $\lambda_2$, hence defines an effectively optimal set of momentum parameters without prior knowledge of $\lambda_2$. This strategy is based on the approaches in \cite{austin2024dynamically,cowal2025faber}. 

\begin{remark}
We proceed under the following two assumptions. First, we assume $\lambda_1 > \lambda_2 > 0$ are real, which is a restriction on our problem class.
Second, we assume the residual convergence rate, denoted $d_{k}/d_{k-1}$ where 
$d_k = \|{v_{k+1} - \nu_k x_k}\|_2$, see Algorithms \ref{alg:power-method} and \ref{alg:dynamic-generalized-momentum}, agrees with the analytical convergence rate of the approximate eigenvector to the first eigenmode, as in Theorem \ref{thm:general-momentum-convergence}.

In \cite{austin2024dynamically}, the convergence of the analogous lower-order (Chebyshev-type) dynamic momentum method is rigorously established for symmetric matrices. Here, as in \cite{cowal2025faber}, we do not conduct a formal convergence analysis of the dynamic generalized momentum power method, but rather provide a heuristic justification based on the convergence analysis for Algorithm \ref{alg:generalized-momentum} in the preceding section, and a proof of the key contraction mechanism behind the convergence in Lemma \ref{lem:dynamic-contraction-map}. We further demonstrate the effectiveness of the method numerically in Section \ref{sec:numerics}.
\end{remark}
Theorem \ref{thm:general-momentum-convergence} shows the convergence rate of Algorithm \ref{alg:generalized-momentum} is given by 
\begin{align}\label{eqn:gen-crate}
(1 + \sigma^{-1}\sqrt{2 \varepsilon_\ast})^{-1} \approx \exp\left(-\sigma^{-1}\sqrt{2 \varepsilon_\ast} \right),
\end{align}
where the approximation holds up to higher-order terms in $\varepsilon_*$.
From \eqref{eqn:gen-crate} and Corollary \ref{cor:l2opt}, we can define an asymptotically optimal convergence rate by
\begin{align}\label{eqn:rhodef}
\rho = \exp(-\sigma^{-1} \sqrt{2 \varepsilon}) 
= \exp(-\sigma^{-1} \sqrt{2 (r^{-1} -1)}), ~\text{with}~
\varepsilon = \lambda_1/\lambda_2-1, 
\end{align}
and $r = \lambda_2/\lambda_1$. Inverting the expression $\rho(r)$ in \eqref{eqn:rhodef} for $r$ as a function of $\rho$ yields 
\begin{align}\label{eqn:rofrho}
r(\rho) = \frac{1}{\frac{\sigma^2}{2}(\log\rho)^2 + 1}.
\end{align}

At each iteration of  Algorithm \ref{alg:dynamic-generalized-momentum}  the residual convergence rate $\rho_k = d_k/d_{k-1}$ is computed. The quantity $\rho_k$ can be thought of as an approximation to the optimal convergence rate $\rho$ given in \eqref{eqn:rhodef}, meaning $\rho_{k} = \rho(\hat r)$, where $\hat r$ is an approximation to $r$. By \eqref{eqn:rofrho} we then solve for $r_{k+1} = \hat r (\rho_k)$, which gives us an approximation to $r = \lambda_2/\lambda_1$. Multiplying by the Rayleigh quotient gives an approximation to $\lambda_2$.  

This procedure dynamically determines the parameters $\beta_1, \dots, \beta_{m-1}$, of Algorithm \ref{alg:generalized-momentum}, yielding Algorithm \ref{alg:dynamic-generalized-momentum}.  In Lemma \ref{lem:dynamic-contraction-map}, we show the map $r(\rho)$ given by \eqref{eqn:rofrho} is a contraction, without which the method could not be expected to work.

\begin{remark}
Other definitions of the asymptotic convergence rate are possible! We also tested $\rho = (1 + \sigma^{-1}\sqrt{2 \varepsilon})^{-1}$, which is equivalent up to higher-order terms in $\varepsilon$, see \eqref{eqn:gen-crate}.  We prefer the choice \eqref{eqn:rhodef} as it gave a better contraction result in Lemma \ref{lem:dynamic-contraction-map}, as well as slightly better convergence in practice.
\end{remark}
\begin{algorithm}
\caption{Dynamic Generalized Momentum Power Method}
\label{alg:dynamic-generalized-momentum}
\begin{algorithmic}[1]
\REQUIRE Matrix $\vect{A} \in \mathbb{R}^{n \times n}$, vector $\vect{v}_0  \in \mathbb{R}^n$, iteration count $N \in \mathbb{N}$, 
and probability vector $p = (p_0,\ldots,p_m)$ satisfying \eqref{meanzero} with $p_1 = 0$. 
\STATE Set $\sigma^2 = \sum_{j=0}^m (j-1)^2p_j$
\STATE Do $m-1$ iterations of power iteration with inputs $p_0\vect{A}$ and $\vec{v}_0$
\FOR{{$k = m-1,\ldots,N-1$}}
\STATE{Set $\vect{v}_{k+1} = \vect{A} \vect{x}_{k}$, $\nu_{k} = \langle \vect{v}_{k+1}, \vect{x}_{k}\rangle$ and $d_{k} = \|\vect{v}_{k+1} - \nu_{k} \vect{x}_{k}\|_2$}
\STATE{Set $\rho_{k-1} = \min\{d_{k}/d_{k-1},1\}$  and $r_k = \left(\frac{\sigma^2}{2}(\log \rho_{k-1})^2 + 1 \right)^{-1}$}

\STATE{Set $\vect{u}_{k+1,0}' = \vect{v}_{k+1}$}
\STATE{Set $H_{k, 0} = 1$}
\FOR{$j = 1, \dots, m-1$}
\STATE{Set $\beta_{k,j} = p_{j+1}p_0^{j}(\nu_{k} r_{k})^{j+1}$}
\STATE{Set $H_{k, j} = H_{k, j-1} h_{k+1-j}$}
\STATE{Set $\vect{u}'_{k+1, j} = \vect{u}'_{k+1,j-1} -  (\beta_{k,j}/H_{k, j}) \vect{x}_{k-j}$}

\ENDFOR
\STATE{Set $\vect{u}_{k+1} = \vect{u}'_{k+1, m-1}$}
\STATE{Set $h_{k+1} = \|\vect{u}_{k+1}\|_2$ and $\vect{x}_{k+1} = h_{k+1}^{-1} \vect{u}_{k+1}$}
\ENDFOR
\RETURN $\vec{x}_N$
\end{algorithmic}
\end{algorithm}

The next lemma establishes the key contraction result. 

\begin{lemma}\label{lem:dynamic-contraction-map}
    Assume that $\rho \in [\rho_1, 1]$ where $({-\sigma^2 + \sigma\sqrt{\sigma^2 + 4}})/{2} < \rho_1 < 1.$  Then the function $r(\rho)$ given by \eqref{eqn:rofrho} is Lipschitz with constant $c<1$ on $[\rho_1,1]$.
\end{lemma}

\begin{proof}
We will show that
$|r(\rho) - r(\rho')| \leq c|\rho - \rho'|,$
where $c < 1$, for some domain of $\rho$ bounded above by $1$.  By the Mean Value Theorem, it is sufficient to establish that $0 \leq \frac{dr}{d\rho} \leq c$ for all $\rho$ on the interval.  Computing the derivative, we find that
\begin{align}\label{eqn:drdrho}
\frac{dr}{d\rho} = \left(\frac{\sigma^2}{2}(\log \rho)^2 + 1\right)^{-2} \cdot \sigma^2(-\log \rho) \cdot \rho^{-1}.
\end{align}
Since every factor is non-negative, we establish that $\frac{dr}{d\rho} \geq 0$.  
Additionally, we observe that $((\sigma^2/2)(\log \rho)^2 + 1)^{-2} \leq 1$.  
Hence,
$0 \leq {dr}/{d\rho} \leq \sigma^2(-\log \rho) \cdot \rho^{-1},$
for all $\rho \leq 1$.  Using the inequality $\rho\log(\rho) \geq \rho - 1$ for all $\rho \in (0,1]$, we have $-\log(\rho)\rho^{-1} \leq (1 - \rho)\rho^{-2}$.  This establishes the upper bound
\begin{align}\label{eqn:dr-upper}
0 \leq \frac{dr}{d\rho} \leq \sigma^2\frac{1 - \rho}{\rho^2}.
\end{align}
The right hand side of \eqref{eqn:dr-upper} is bounded above by 1 whenever
$\rho^2 + \sigma^{2}\rho - \sigma^{2} \geq 0,$
which is guaranteed when $$ \rho \geq \rho_1 > \frac{-\sigma^2 + \sigma\sqrt{\sigma^2 + 4}}{2}.$$
Let $c = \sigma^2(1 - \rho_1)\rho_1^{-2}.$
This allows us to conclude that 
$0 \leq {dr}/{d\rho} \leq c < 1,$
for all $\rho \in [\rho_1, 1]$.
\end{proof}
This key contraction property shows that given an approximation $\rho_k \approx \rho$, the map \eqref{eqn:rofrho} 
produces an approximation $r_{k+1} \approx r$. In particular, for problems with small spectral gaps, meaning $r$ is close to 1, we have from \eqref{eqn:drdrho} that $dr/d\rho$ is close to zero, hence the method is most stable in the regime where it is most useful.

\subsection{Numerical Results} \label{sec:numerics}
The assumptions of Theorem \ref{thm:general-momentum-convergence} imply that Algorithms \ref{alg:generalized-momentum} and \ref{alg:dynamic-generalized-momentum} can be applied to find the dominant eigenvectors of non-symmetric operators whose eigenvalues of largest magnitude are real.  Moreover, the benefits of the momentum-based approach show themselves primarily when the spectral gap is small. 
 In this section, we demonstrate the method on
 example problems that satisfy these criteria.  
 Code to reproduce the numerical experiments presented in this section is available at: 
\begin{center}
\url{https://github.com/petercowal/random-walk-power-methods}.  
\end{center}

\begin{remark}
To account for the possible difference in complex phase between the output of Algorithm \ref{alg:generalized-momentum} and our ground truth vectors, 
we use the following measure of relative error between unit vectors $\hat{\vect{x}}$ and $\hat{\vect{x}}_{\text{true}}$:
$$
    \text{relerr}(\hat{\vect{x}}, \hat{\vect{x}}_{\text{true}}) = \| (\hat{\vect{x}}^*\hat{\vect{x}}_{\text{true}})\hat{\vect{x}} - \hat{\vect{x}}_{\text{true}}\|_2 
    = \sqrt{1 - |\hat{\vect{x}}^*\hat{\vect{x}}_\text{true}|^2} = \sin(\theta),
$$
where $\theta$ is the angle between $\text{span}(\hat{\vect{x}})$ and $\text{span}(\hat{\vect{x}}_{\text{true}})$.

For our numerical experiments, unless stated otherwise, we use SciPy's {\tt eigs} function with default precision to determine a reference dominant eigenvector $\vect{x}_{\text{true}}$, as well as the two dominant eigenvalues $\lambda_1$ and $\lambda_2$, which we use to compute the spectral gap.
\end{remark}

\subsubsection{Toy example}

Consider the matrix 
$$
\vect{A} = \begin{bmatrix}
    1.01 & & & \\ & 1 & & \\ & & & -1/2 \\ & & 1/2 &
\end{bmatrix}
$$ 
with eigenvalues $\lambda_1 = 1.01$, $\lambda_2 =1$, $\lambda_{3} = i/2$ and $\lambda_4 = -i/2$ and a spectral gap of $\varepsilon = 0.01$.  Note that $\lambda_3$ and $\lambda_4$ lie within the 4-hypocycloid and 5-hypocycloid, but not the 3-hypocycloid (see Figure \ref{fig:hypocycloids}).  As such, Theorem \ref{thm:general-momentum-convergence} only predicts that momentum methods based on the 4- and 5-hypocycloids will converge at an accelerated rate.  We run Algorithm \ref{alg:generalized-momentum} with momentum parameters determined by Table \ref{tab:three-term-parameters} and plot the results in Figure \ref{fig:general-momentum}.  
The results indicate that the 4- and 5-hypocycloid methods converge at the predicted rate (and significantly faster than the power method), whereas the deltoid method converges relatively slowly and the Chebyshev method fails to converge at all.

We use the same toy example to verify the convergence of Algorithm \ref{alg:dynamic-generalized-momentum}, with results shown on the right of Figure \ref{fig:general-momentum}.  The corresponding probability distributions are given in Table \ref{tab:dynamic-parameters}.  Notably, the mixed probability distribution performs the best.

\begin{table}[h!]
    \centering
    \begin{tabular}{c|llllll|c}
        Name & $p_0$ & $p_2$ & $p_3$ & $p_4$ & $p_5$ & $p_6$ & $2\sigma^{-2}$\\ \hline
        dynamic 2   & $1/2$ & $1/2$ &       &       &       & & $2$ \\
        dynamic 3   & $2/3$ &       & $1/3$ &       &       & & $1$ \\
        dynamic 4   & $3/4$ &       &       & $1/4$ &       & & $2/3$\\
        dynamic 5   & $4/5$ &       &       &       & $1/5$ & & $1/2$ \\
        dynamic 6   & $5/6$ &       &       &       &  & $1/6$ & $2/5$  \\
        dynamic 2-3 & $7/12$& $1/4$ & $1/6$ &       &       & & $4/3$ \\
        dynamic 2-4 & $5/8$ & $1/4$ &       & $1/8$ &       & & $1$\\
    \end{tabular}
    \caption{The dynamic parameters for the momentum methods plotted in Figure \ref{fig:general-momentum} (right) as well as Figure \ref{fig:real-examples}}
    \label{tab:dynamic-parameters}
\end{table}

\begin{figure}[h!]
    \centering
    \includegraphics[width=0.49\linewidth]{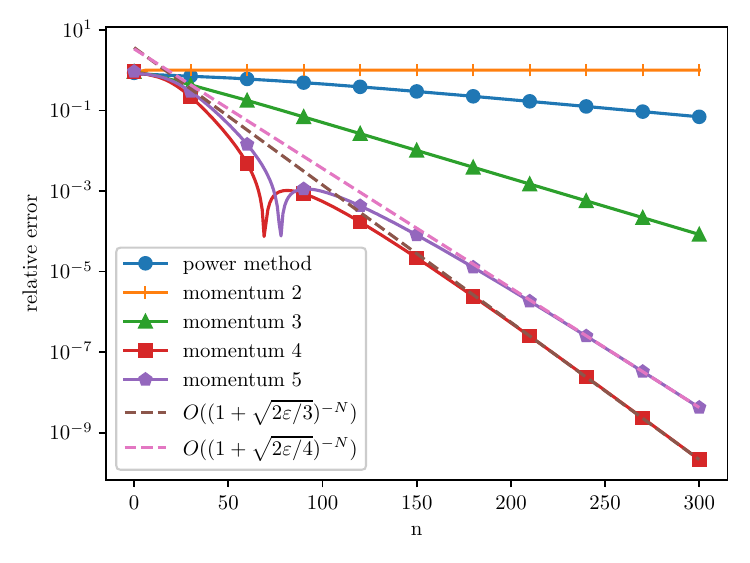}
    \includegraphics[width=0.49\linewidth]{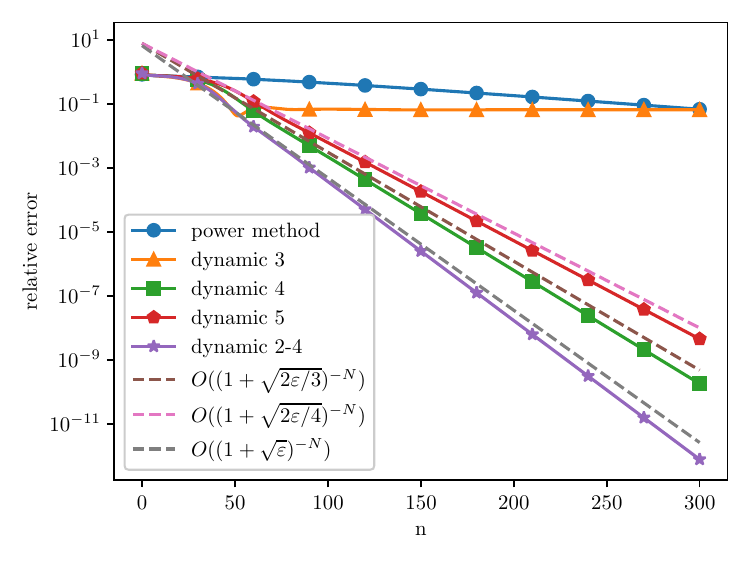}
    \caption{Left: convergence of hypocycloid momentum methods for varying degrees of hypocycloid for the toy example.  Right: convergence of the dynamic momentum method with various probability distributions.}
    \label{fig:general-momentum}
\end{figure}

\subsubsection{Barbell example}

We construct a random directed barbell graph with $2N$ vertices, separated into two halves of $N$ vertices with a single bidirectional edge bridging both halves.  Within each half of the graph, we randomly connect vertices with probability $p$. Specifically,
$$\vec{A} = \begin{bmatrix}
a_{11} & \cdots & a_{1N} \\
\vdots & & \vdots \\
a_{N1} & \cdots & a_{NN} & 1
\\ & & 1 & b_{11} & \cdots & b_{1N} \\
& & & \vdots & & \vdots \\
& & & b_{N1} & & b_{NN}
\end{bmatrix},$$ 
where $P(a_{ij} = 1) = P(b_{ij} = 1) = p$ and $P(a_{ij} = 0) = P(b_{ij} = 0) = 1- p$.
The adjacency matrix of this graph will have a dominant real eigenvalue and a small spectral gap, due to the barbell geometry. Moreover, by the circular law, we expect it to have complex eigenvalues that are roughly uniformly distributed within a disc in the complex plane with radius $1/\sqrt{Np}$ centered at zero.

For our specific example, we consider $N = 1000$ and $p = 1/250$, for an expected spectral radius of $1/2$.  Since the graph is random, some eigenvalues fall slightly outside this radius.
The 4-hypocycloid is the smallest hypocycloid that encloses a disk centered at zero of radius $1/2$, which it does exactly; higher degree hypocycloids enclose a disk of radius greater than $1/2$.  We plot the eigenvalues of the adjacency matrix alongside hypocycloids as well as the results of applying the dynamic momentum method to the adjacency matrix in Figure \ref{fig:barbell-results}.  Note that the dynamic methods converge exactly when the spectrum is enclosed by the corresponding hypocycloid, at a rate much faster than the unaccelerated power method.

\begin{figure}[h!]
    \centering
    \includegraphics[width=0.44\linewidth]{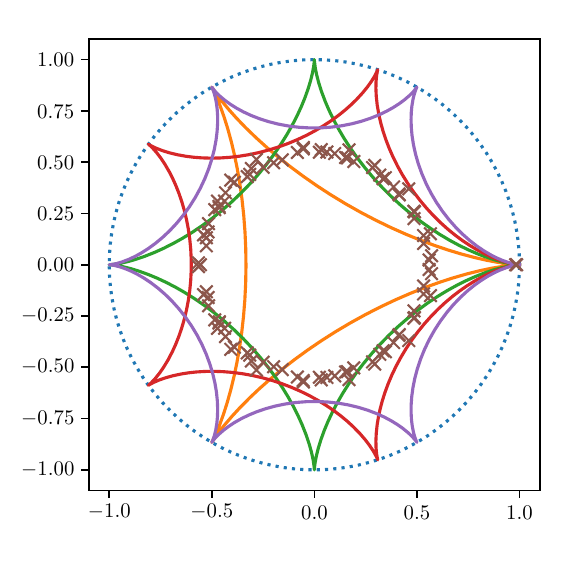}
    \includegraphics[width=0.55\linewidth]{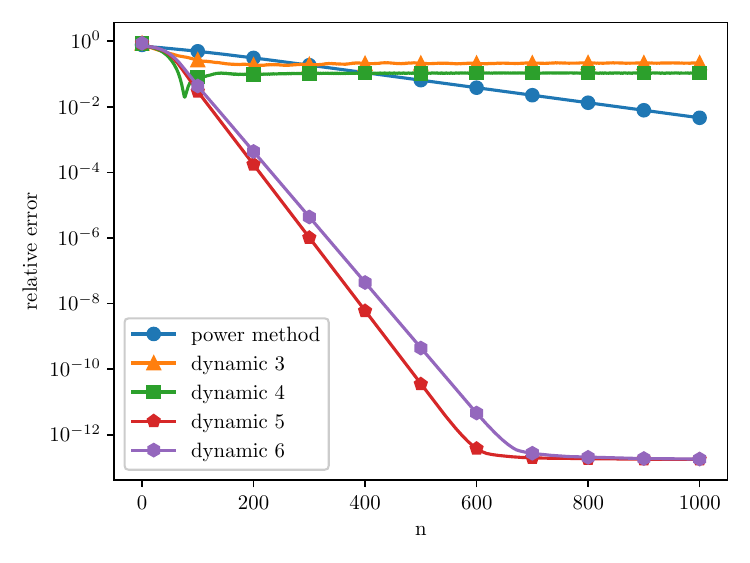}
    \caption{Left: the 100 largest eigenvalues (marked with a cross) of the barbell adjacency matrix ($N = 1000$, $p = 1/250$), alongside hypocycloids of degrees 3, 4, 5, 6.  Right: convergence of the dynamic momentum method applied to the barbell problem.}
    \label{fig:barbell-results}
\end{figure}

We generate a larger barbell adjacency matrix $(N = 16000, p = 1/4000)$ with a similar spectral profile, with results shown in Figure \ref{fig:big-barbell-results}.  Given this matrix's large size and small spectral gap, we found that {\tt eigs} failed to produce a solution in a reasonable time, so we compute a reference eigenvector using Algorithm \ref{alg:dynamic-generalized-momentum} with parameters $p_0 = 4/5$, $p_5 = 1/5$ with 2000 iterations.

\begin{figure}[h!]
    \centering
    \includegraphics[width=0.55\linewidth]{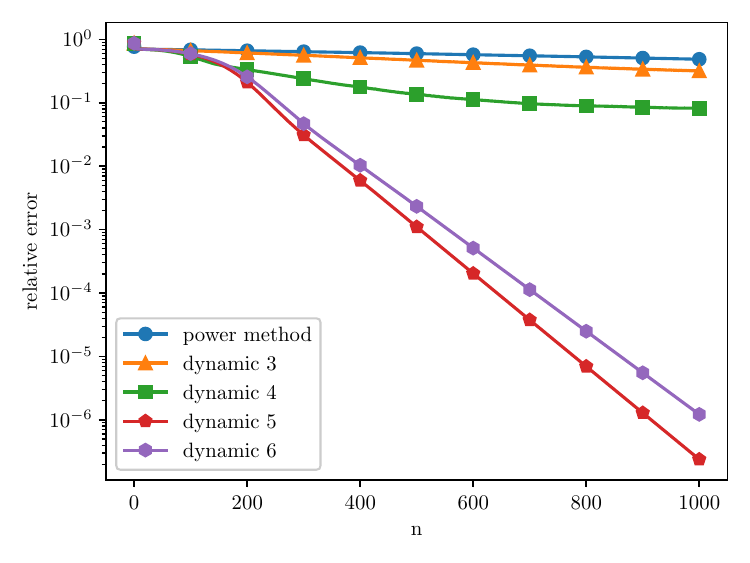}
    \caption{Convergence of the dynamic momentum method applied to a larger barbell matrix ($N = 16000$, $p = 1/4000$).}
    \label{fig:big-barbell-results}
\end{figure}

\subsubsection{Application to sparse matrices}
It is worth emphasizing that the convergence criterion for Algorithm \ref{alg:generalized-momentum} requires some a priori knowledge of the spectrum of the matrix.  The dynamic nature of Algorithm \ref{alg:dynamic-generalized-momentum} reduces this requirement somewhat by estimating $\lambda_2$ on the fly, but still requires knowledge of the shape of the spectrum.  

There are instances, however, where a matrix with an unknown spectrum is likely to satisfy our spectrum requirements.  Consider a matrix with real dominant eigenvalues, but whose smaller eigenvalues have small imaginary components.  Such a matrix does not satisfy the convergence criteria for algorithms designed for symmetric matrices.  However, the spectrum of such a matrix is likely to lie within $\lambda_2\Gamma$ when, for example, $p$ is an interpolated combination of $(1/2, 0, 1/2)$ and either $(2/3, 0, 0, 1/3)$ or $(3/4, 0, 0, 0, 1/4)$.  With such a choice of $p$, Algorithm \ref{alg:dynamic-generalized-momentum} can be viewed as a method to decrease the sensitivity of the dynamic momentum method of \cite{austin2024dynamically} to small non-symmetric perturbations.

To verify the convergence rate of Algorithm \ref{alg:dynamic-generalized-momentum} and demonstrate its potential applicability, we use the SuiteSparse Matrix Collection web interface \cite{kolodziej2019suitesparse} to find non-symmetric matrices in the collection \cite{davis2011university} with strong diagonal weights, as such matrices may have real dominant eigenvalues and therefore satisfy the criteria for our method.  In at least three instances, Algorithm \ref{alg:dynamic-generalized-momentum} outperforms both the power method and the first-order dynamic momentum method of \cite{austin2024dynamically}.  A combination of first-order and second-order momentum terms (dynamic 2-3) generally outperforms the methods corresponding to pure hypocycloids; see Figure \ref{fig:real-examples}.

When our dynamic momentum methods are applied to {\tt hor\_131}, the convergence exceeds the theoretical rate, likely due to a larger spectral gap.  For the matrix {\tt windtunnel\_evap\_3d}, the lower imaginary modes appear to feature prominently enough to interfere with the convergence of the first-order dynamic momentum method: not enough to prevent convergence altogether, but enough to slow it down.  

\begin{figure}[h!]
    \centering
    \includegraphics[width=0.49\linewidth]{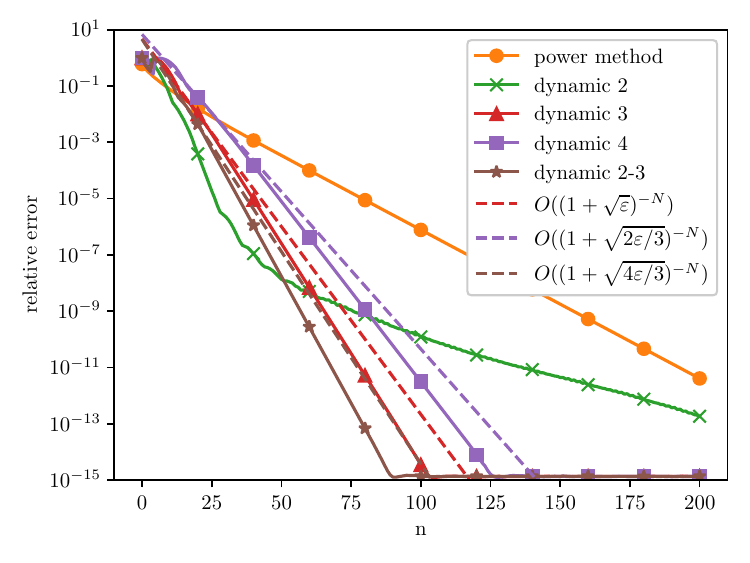}
    \includegraphics[width=0.49\linewidth]{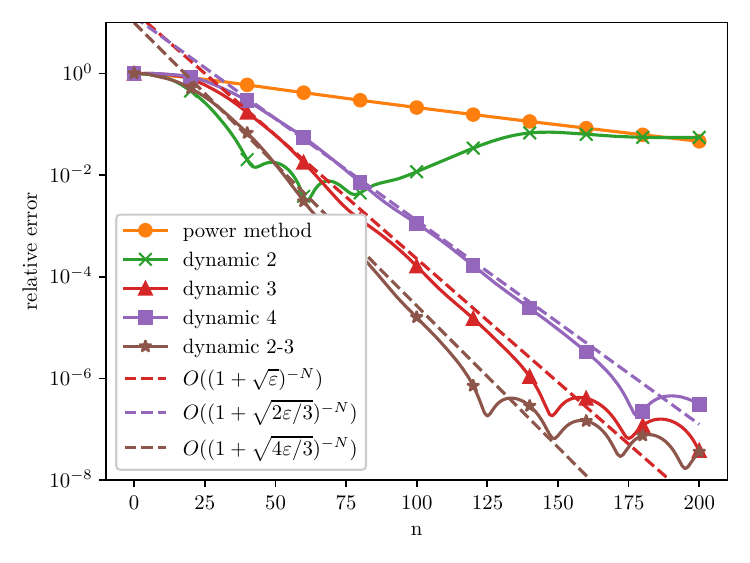}
    
    \includegraphics[width=0.49\linewidth]{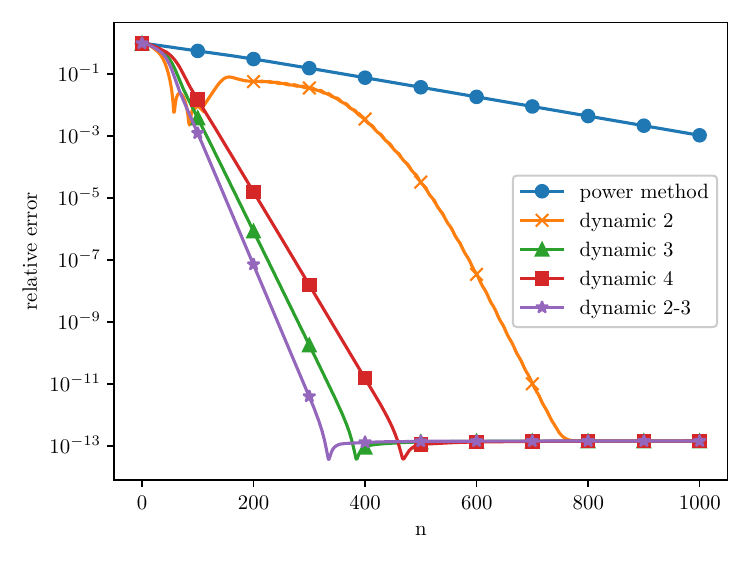}
    \caption{Convergence of the dynamic momentum method applied to three $N \times N$ sparse matrices from \cite{davis2011university}.  Left: {\tt hor\_131}, $N = 434$.  Right: {\tt young1c}, $N = 841$. Bottom: {\tt windtunnel\_evap\_3d}, $N = 40816$.  The probability parameters are given in Table \ref{tab:dynamic-parameters}.}
    \label{fig:real-examples}
\end{figure}

\section{Proof of main results}

\subsection{Proof of Theorem \ref{thm:charcurve}}
\label{sec:proofofthmcharcurve}
The proof of Theorem \ref{thm:charcurve} is divided into two lemmas, which together directly establish the result. First, in Lemma \ref{locationcusps} we characterize the location of the cusps. Second, in Lemma \ref{lem:smoothradiallyconvex}, we show that the curve is smooth except at the cusps and defines a radially convex region. Recall that a cusp for a parametric curve in the complex plane is a point where the derivative (of both real and imaginary parts) is zero, and the derivative in the direction of the tangent changes sign. 

\begin{lemma} \label{locationcusps}
Suppose that $p = (p_0,\ldots,p_m)$ is a probability vector
satisfying the mean-zero condition 
\eqref{meanzero} with $p_1 = 0$ and $p_j > 0$ for some $j \ge 3$. Define
\begin{equation} \label{def:curvelem}
z(t) = \sum_{j=0}^m p_j e^{i(1-j)t}, \quad \text{for} \quad t \in [0,2\pi].
\end{equation}
Then,
$
z'(t) = 0,
$
if and only if 
$t = 2\pi k/n$ for $k\in \{0,\ldots,n-1\}$, where $n$ is the greatest common divisor 
$$
n = \gcd \{j \in \{2,\ldots,m\} : p_j > 0\}.
$$
Moreover, at each point where $z'(t) = 0$, there is a cusp.
\end{lemma}

We illustrate Lemma \ref{locationcusps} in Figure \ref{figwherearecusps}.

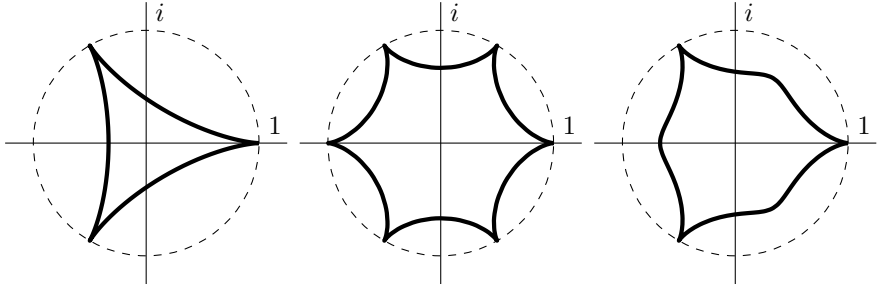
\begin{figure}[h!]
\centering
\begin{tikzpicture}[scale=1.5]
\draw (-1.25, 0) -- (1.25, 0);
\draw (0, -1.25) -- (0, 1.25);
\draw (1, 0) node[anchor=south west]{$1$};
\draw (0, 1) node[anchor=south west]{$i$};
\draw [dashed] (0,0) circle (1);
\draw [ultra thick, domain = 0:360, samples = 60, line join=round]
    plot[smooth] ({2/3*cos(\x) + 1/3*cos(-2*\x)}, {2/3*sin(\x) + 1/3*sin(-2*\x)});
\end{tikzpicture}
\begin{tikzpicture}[scale=1.5]
\draw (-1.25, 0) -- (1.25, 0);
\draw (0, -1.25) -- (0, 1.25);
\draw (1, 0) node[anchor=south west]{$1$};
\draw (0, 1) node[anchor=south west]{$i$};
\draw [dashed] (0,0) circle (1);
\draw [ultra thick, domain = 0:360, samples = 60, line join=round]
    plot[smooth] ({5/6*cos(\x) + 1/6*cos(-5*\x)}, {5/6*sin(\x) + 1/6*sin(-5*\x)});
\end{tikzpicture}
\begin{tikzpicture}[scale=1.5]
\draw (-1.25, 0) -- (1.25, 0);
\draw (0, -1.25) -- (0, 1.25);
\draw (1, 0) node[anchor=south west]{$1$};
\draw (0, 1) node[anchor=south west]{$i$};
\draw [dashed] (0,0) circle (1);
\draw [ultra thick, domain = 0:360, samples = 60, line join=round]
    plot[smooth] ({(1/2)*(5/6*cos(\x) + 1/6*cos(-5*\x)) + (1/2)*(2/3*cos(\x) + 1/3*cos(-2*\x))}, {(1/2)*(5/6*sin(\x) + 1/6*sin(-5*\x)) + (1/2)*(2/3*sin(\x) + 1/3*sin(-2*\x))});
\end{tikzpicture}
\caption{Left to right: we plot \eqref{def:curvelem} for
$p = (\frac{2}{3},0,0,\frac{1}{3})$, $p=(\frac{5}{6},0,0,0,0,0,\frac{1}{6})$,
$p = (\frac{9}{12},0,0,\frac{1}{6},0,0,\frac{1}{12})$. By Lemma \ref{locationcusps} these curves have $3$, $6$, and $3 = \gcd(3,6)$ cusps, respectively.} \label{figwherearecusps}
\end{figure} 

\begin{proof}[Proof of Lemma \ref{locationcusps}] First, we assume that $t_0 = 2 \pi k/n$ for $n =\text{gcd}(\{j \in \{2,\ldots,m\} : p_j > 0\})$, and will show that $z'(t_0)  = 0$. Observe that
$$
z'(t_0) = \sum_{j=0}^m p_j i(1-j) e^{i(1 -j) t_0} = ie^{i t_0} \sum_{j=0}^m p_j (1-j) e^{-i j t_0}.
$$
By the choice of $t_0$, for each fixed $j$, we have $j t_0 = 2\pi \ell$ for some integer $\ell$ whenever $p_j > 0$. 
As a consequence, $e^{-i j t_0} = 1$ when $p_j \neq 0 $, which implies
$$
z'(t_0) = i\sum_{j=0}^m p_j (1-j) = 0,
$$
where the final equality follows from the assumption $\sum_{j=0}^m (1-j) p_j = 0$.

Second, we assume that $z'(t) = 0$ and will show that $t = 2\pi k/n$ where $n = \text{gcd}(\{j \in \{2,\ldots,m\} : p_j > 0\})$.
We have
$$
z'(t) = i e^{i t} \sum_{j=0}^m p_j (1-j) e^{- i j t} = 0
\quad
\implies
\quad
p_0 = \sum_{j=2}^m p_j (j-1) e^{-i j t}.
$$
Note that by assumption $p_0 > 0$ and $p_j(j-1) \ge 0$ for $j \in \{2,\ldots,m\}$. By the triangle inequality
$$
\left| \sum_{j=2}^m p_j (j-1) e^{-i j t} \right| \le 
 \sum_{j=2}^m \left| p_j (j-1) e^{-i j t} \right| = \sum_{j=2}^m p_j (j-1) = p_0,
$$
where the final equality follows from the assumption that $\sum_{j=0}^m (1-j) p_j = 0$. The triangle inequality for complex numbers only holds with equality when all of the complex numbers have the same phase. Moreover, since $p_0 > 0$, the phase must be $0 \mod 2\pi$. This only happens when $j t = 2\pi \ell$ for some integer $\ell$. We need this to happen for all $j \ge 2$ such that $p_j > 0$. Hence, $t = 2 \pi k/n$ for some integer $k$ where $n = \text{gcd}(\{j \in \{2,\ldots,m\} : p_j > 0\})$.

It remains to show that there is a cusp at each of the points where $z'(t) = 0$. By symmetry, it suffices to show that there is a cusp at $t_0 = 0$. We will show that 
\begin{equation} \label{eq:cuspformula}
z(t) = 1 + a t^2 +i b t^3 + \mathcal{O}(t^4), \quad \text{as} \quad t \to 0,
\end{equation}
where $a,b \in \mathbb{R}$ are nonzero real numbers, that is, near zero, the curve in the complex plane acts like the parametric curve $(1 + a t^2,b t^3) \in \mathbb{R}^2$, which has a cusp when $t=0$. Expanding $z(t)$ in Taylor series at $0$ gives
$$
a = \frac{z''(0)}{2} = -\frac{1}{2}\sum_{j=0}^m (1-j)^2 p_j, \quad \text{and} \quad i b = \frac{z'''(0)}{6} = -\frac{i}{6} \sum_{j=0}^m (1-j)^3 p_j.
$$
The mean-zero condition 
\eqref{meanzero}  and the assumption $p_j > 0$ for some $j \ge 3$ ensure that $a$ and $b$ are nonzero,  so \eqref{eq:cuspformula} is established, and the proof is complete.
\end{proof}

\begin{lemma} \label{lem:smoothradiallyconvex}
Suppose that $p = (p_0,\ldots,p_m)$ is a probability vector
satisfying the mean-zero condition \eqref{meanzero} with $p_1 = 0$ and $p_j >0$ for some $j \ge 3$. Define
$$
z(t) = \sum_{j=0}^m p_j e^{i(1-j)t}, \quad \text{for} \quad t \in [0,2\pi].
$$
Then, $z(t)$ encloses a radially convex set.
\end{lemma}

\begin{proof}[Proof of Lemma \ref{lem:smoothradiallyconvex}]
Suppose that $p_j > 0$ for some $j \ge 3$. Then,
$$
|z(t)| = \left| \sum_{j=0}^m p_j e^{i(1-j)t} \right| \ge p_0 - \sum_{j=2}^m p_j = \sum_{j=2}^m (j-2) p_j > 0,
$$
where the final equality follows from \eqref{meanzero}, and the final inequality uses the fact that $p_j > 0$ for some $j \ge 3$.
Let $p^{(k)}$ be the probability vector defined by $p_0^{(k)} = (k-1)/k$ and $p^{(k)}_k = 1/k$  associated with the $k$-cusp hypocycloid (see Section \ref{sec:hypocycloid}). Assume $p^{(k)}$ is padded by zeros so that it has the same length as $p$.
Then, we can write
$$
p = \sum_{k=2}^m w_k p^{(k)},
$$
where $w_k = k p_k$ for $k = 2,\ldots,m$. Indeed, for $j \ge 2$ we have
$$
\sum_{k=2}^m w_k p^{(k)}_j = w_j p^{(j)}_j = j p_j \frac{1}{j} = p_j,
$$
and 
$$
\sum_{k=2}^m w_k p^{(k)}_0 = \sum_{k=2}^m k p_k \frac{k-1}{k} = \sum_{k=2}^m (k-1) p_k = p_0,
$$
where the final equality follows from \eqref{meanzero}. By definition, the weights $w_k$ are nonnegative, and we have
$$
\sum_{k=2}^m w_k = \sum_{k=2}^m k p_k = \sum_{k=0}^m p_k + \sum_{k=0}^m (k-1) p_k = 1, 
$$
where the final equality follows from \eqref{meanzero} and the fact that $\sum_{k=0}^m p_k = 1$. By linearity, we have
\begin{equation} \label{eq:convexcomb}
z(t) = \sum_{k=2}^m w_k z^{(k)}(t),
\end{equation}
where $z^{(k)}(t)$ is the $k$-cusp hypocycloid curve defined by \eqref{eq:gamma-higher-order}. Note that we have already shown that $|z(t)| > 0$, which means that we can  write $z(t)$ in polar form as
$$
z(t) = r(t) e^{i \theta(t)},
$$
for $r(t) > 0$ and $\theta(t) \in [0,2\pi]$. 

To show that $z(t)$ encloses a radially convex domain, it suffices to show that in this polar form $\theta'(t) \ge 0$. Taking the logarithmic derivative of $z(t)$ yields 
$$
\frac{z'(t)}{z(t)} 
= \frac{r'(t)}{r(t)} + i \theta'(t),
$$
so it follows that $\Im (z'(t)/z(t)) = \theta'(t)$, where $\Im$ denotes the imaginary part. Since
$$
\Im \left( \frac{z'(t)}{z(t)} \right) = \Im \left( \frac{z'(t) \overline{z(t)}}{|z(t)|^2} \right), 
$$
we will show that
$$
\Im \left( z'(t) \overline{z(t)} \right) \ge 0.
$$
Expanding $z(t)$ as a convex combination
\eqref{eq:convexcomb}
of $k$-cusp hypocycloid curves
 gives
\begin{equation*}
\Im \left( z'(t) \overline{z(t)} \right) 
= \sum_{j,k=2}^m w_j w_k \Im \left( {z^{(j)}}'(t) \overline{z^{(k)}(t)} \right) ,
\end{equation*}
which can be expanded as 
\begin{multline} \label{eq:expandsumdiagoff}
 \sum_{j,k=2}^m w_j w_k \Im \left( {z^{(j)}}'(t) \overline{z^{(k)}(t)} \right) 
= \sum_{j=2}^m w_j^2 \Im \left( {z^{(j)}}'(t) \overline{z^{(j)}(t)} \right) \\
+ \sum_{2 \le j < k \le m}
w_j w_k
\Im \left( {z^{(j)}}'(t) \overline{z^{(k)}(t)} +
{z^{(k)}}'(t) \overline{z^{(j)}(t)} \right).
\end{multline}
To complete the proof, it suffices to show that
$$
\Im \left( {z^{(j)}}'(t) \overline{z^{(k)}(t)} +
{z^{(k)}}'(t) \overline{z^{(j)}(t)} \right) \ge 0,
\quad \forall j,k \in \{2,\ldots,m\}.
$$
This inequality implies that all the terms in both of the sums on the right-hand side of 
\eqref{eq:expandsumdiagoff} are nonnegative.
Using the definition  \eqref{eq:gamma-higher-order} of $z^{(j)}(t)$,
we have
$$
z^{(j)}(t) = \frac{j-1}{j} e^{i t} + \frac{1}{j} e^{(1-j) i t}.
$$
It is then straightforward to verify that 
\begin{multline} \label{eq:termsnonneg}
\Im \left( {z^{(j)}}'(t) \overline{z^{(k)}(t)} +
{z^{(k)}}'(t) \overline{z^{(j)}(t)} \right) = 
\frac{(j-1)(k-2)}{j k} \Big(1 - \cos( k t) \Big)  + \\ \frac{(j-2)(k-1)}{j k}\Big(1 - \cos(j t) \Big) + \frac{j +k -2}{j k}\Big(1 - \cos((k - j)t) \Big) .
\end{multline}
Since each of the three terms on the right-hand side is nonnegative for $j,k \in \{2,\ldots,m\}$, the proof is complete. 
\end{proof}

\subsection{Proof of Theorem \ref{thm:boundedrapid}}
\label{sec:proofthmboundedrapid}
The proof of Theorem \ref{thm:boundedrapid} is divided into six lemmas.
In Lemmas \ref{lem:roughlocation}, \ref{thm:root-estimate}, \ref{lemma:root-asymptotics}, and \ref{lem:trivialgrowthrate}, we establish some supporting results. In Lemma \ref{lem:growth-actually-happens} and \ref{lem:showboundedness}, we prove growth and boundedness results, respectively, which together establish Theorem \ref{thm:boundedrapid}.

\begin{lemma} \label{lem:roughlocation}
Suppose that $p = (p_0,\ldots,p_m)$ is a probability vector
satisfying the mean-zero condition \eqref{meanzero} with $p_1 = 0$ and $p_j > 0$ for some $j \ge 3$.  Let $Q_z(r)$ be the characteristic polynomial introduced in \eqref{eq:charpolyintro} defined by
$$
Q_z(r) = r^m - \frac{z}{p_0} r^{m-1} + \sum_{j=2}^m \frac{p_j}{p_0} r^{m-j},
$$
and let $\Gamma$ be the closure of the region enclosed by the curve parameterized by $z(t)$, $t \in [0,2\pi]$ defined in \eqref{eq:curveformula}.
Then, $Q_z(r)$ has exactly one root of magnitude strictly greater than $1$ when $z \in \mathbb{C} \setminus \Gamma$ and has all roots of magnitude strictly less than $1$ when $z \in \interior(\Gamma)$, where $\interior(\Gamma)$ denotes the interior of $\Gamma$.
\end{lemma}

\begin{proof}[Proof of Lemma \ref{lem:roughlocation}]
From \eqref{recurrence} with $p_1 = 0$ we have
\begin{equation} \label{recurrenceformulainlemma}
P_{k+1}(z) = \frac{z}{p_0} P_k(z) - \sum_{j=2}^m \frac{p_j}{p_0} P_{k+1-j}(z), \quad \text{for} \quad k \ge m-1.
\end{equation}
We can write this as a linear first-order difference equation on $ m$-dimensional vectors as follows:
$$
\begin{pmatrix}
P_{k+1}(z) \\
P_k(z) \\
\vdots \\
P_{k-m+3}(z) \\ 
P_{k-m+2}(z)
\end{pmatrix}
=
\vec{T}(z)
\begin{pmatrix}
P_{k}(z) \\
P_{k-1}(z) \\
\vdots \\
P_{k-m+2}(z) \\
P_{k-m+1}(z)
\end{pmatrix},
$$
where $\vec{T}(z)$ is the $m \times m$ matrix 
\begin{align}\label{eqn:Asys}
\vec{T}(z) := \begin{pmatrix}
\frac{z}{p_0} &
-\frac{p_2}{p_0} &
\cdots &
-\frac{p_{m-1}}{p_0} &
-\frac{p_m}{p_0} \\
1 &  &  &  &  \\
 & 1 &  &  &  \\
 &  & \ddots &  &  \\
 &  &  &  1 &  \\
\end{pmatrix}.
\end{align}
When $z$ is sufficiently large, by the Gershgorin circle theorem the matrix
$\vec{T}(z)$ from \eqref{eqn:Asys}
has $1$ eigenvalue of magnitude strictly greater than $1$, and $m-1$  eigenvalues of magnitude at most $1$.
The characteristic polynomial of this matrix (also called the characteristic polynomial of the recurrence) is
$$
Q_z(r) = r^m - \frac{z}{p_0} r^{m-1} + \sum_{j=2}^m \frac{p_j}{p_0} r^{m-j}.
$$
To see that $Q_z(r)$ is the characteristic polynomial of $\vec{T}(z)$, note that it is a $90$-degree rotation of the Companion Matrix for the polynomial $Q_z(r)$. Suppose that $Q_z$ has a root of magnitude $1$. Then,
$$
e^{i t m} - \frac{z}{p_0} e^{i t (m-1)} +\sum_{j=2}^m \frac{p_j}{p_0} e^{i t (m-j)} = 0.
$$
Solving for $z$ gives
\begin{equation} \label{condswitch}
z = \sum_{j=0}^m p_j e^{it(1-j)}.
\end{equation}
The zeros of $Q_z(r)$ are continuous functions of $z$. Thus, the magnitude of the roots can only switch from strictly less than $1$ to strictly more than $1$ when \eqref{condswitch} is satisfied, which is an equation for the curve $\gamma$. By a continuity argument, for $z$ outside of this curve, we have exactly one eigenvalue of magnitude strictly greater than one. 

To complete the proof, we demonstrate that when $z=0$ all of the eigenvalues of the matrix are strictly less than $1$.  Using the assumption that at least one $p_j > 0$ for $j \geq 3$, we have
$$p_0 = \sum_{j=2}^m (j-1)p_j > \sum_{j=2}^m p_j$$
which implies
$$ \sum_{j=2}^m p_j/p_0 < 1.$$
Now, for $z = 0$, consider 
$$Q_0(r)r^{-m} = 1 + \sum_{j=2}^m \frac{p_j}{p_0}r^{-j}.$$
If $Q_0(r) = 0$, then 
$$\left| \sum_{j=2}^m \frac{p_j}{p_0} r^{-j} \right| = 1.$$
But if $|r| \geq 1$, then
\begin{equation} \label{eq:keyineqtouse}
\left| \sum_{j=2}^m \frac{p_j}{p_0} r^{-j} \right| \leq  \sum_{j=2}^m \frac{p_j}{p_0} \left| r^{-j} \right| \leq \sum_{j=2}^m \frac{p_j}{p_0} < 1.
\end{equation}
Hence, we conclude that $Q_0(r) = 0$ only if $|r| < 1$.  
\end{proof}

\begin{lemma} \label{thm:root-estimate}
Let $z = 1 + \varepsilon$ for $\varepsilon > 0$.  Then, the characteristic polynomial $Q_z(r)$ defined in  \eqref{eq:charpolyintro}
for a probability vector $p = (p_0,\ldots,p_m)$ that satisfies \eqref{meanzero} with $p_1 = 0$,
has a real root $r$ satisfying
$$r \geq 1 + \sigma^{-1}\sqrt{2\varepsilon},$$
where $\sigma$ is given by \eqref{eq:variance}.  
\end{lemma}

\begin{proof}[Proof of Lemma \ref{thm:root-estimate}]
When $z = 1+\varepsilon$, by  Lemma \ref{lem:roughlocation}, we know that $Q_z$ has exactly one root of magnitude strictly greater than $1$, which determines the behavior of $P_n(z)$ as $n \to \infty$. Since $P_n(z)$ is real valued for $z=1+\varepsilon$, it follows that the root of magnitude strictly greater than $1$ is real. By \eqref{eq:conformalmap1} and the assumption $p_1 = 0$ any root $r$ of the characteristic polynomial $Q_z$ satisfies
\begin{equation} \label{eq:conditionforrootqn}
z = p_0 r  + \sum_{j=2}^m p_j r^{1 - j}.
\end{equation}
We now consider the Taylor series for $z$ as a function of $r$ centered at $r_0 = 1$.  The first coefficient is given by
\begin{equation} \label{taylora0}
a_0 = \sum_{j=0}^m p_j = 1.
\end{equation}
To find the second and third coefficients, note that
$$\frac{dz}{dr} = p_0 + \sum_{j=2}^m p_j (1 - j) r^{ - j},$$
and 
\begin{equation} \label{taylora1}
a_1 = \frac{dz}{dr}(1) =  p_0 + \sum_{j=2}^m p_j (1 - j) = \sum_{j=0}^m p_j (1 - j) = 0.
\end{equation}
Additionally,
\begin{equation} \label{taylora2}
\frac{d^2z}{dr^2} = \sum_{j=2}^m p_j j(j - 1) r^{ - (j + 1)},
\end{equation}
so 
$$a_2 = \frac{1}{2} \frac{d^2z}{dr^2}(1) =  \frac{1}{2} \sum_{j=2}^m p_j j(j - 1) > 0.$$
To establish an upper bound on the Taylor series approximation, we also compute
$$\frac{d^3z}{dr^3} = \sum_{j=2}^m -p_j j(j - 1)(j + 1) r^{ - (j + 2)},$$
which is negative for all $r > 0$.  Hence, the Taylor series up to quadratic terms of the characteristic polynomial serves as an upper bound for $z$, namely
$$ z \leq 1 + \left( \sum_{j=2}^m \frac{j(j - 1)}{2}p_j \right) (r-1)^2.$$
This upper bound holds for all $r \geq 1$.  We now set $z = 1 + \varepsilon$ and find that
$$1 + \varepsilon \leq 1 + \left( \sum_{j=2}^m \frac{j(j - 1)}{2}p_j \right) (r-1)^2, $$
and solving for $r$ under the restriction that $r > 1$ gives us a lower bound for the dominant zero
$$r \geq 1 + \sqrt{C_P\varepsilon},$$
where
$$C_P = \left( \sum_{j=2}^m \frac{j(j - 1)}{2}p_j \right)^{-1}.$$
To rewrite $C_P$ in terms of $\sigma$, we first extend the sum to $j=0,1$ since it is zero at those values, by which
$$
\sum_{j=2}^m \frac{j(j-1)}{2}p_j
= \sum_{j=0}^m \frac{j(j-1)}{2}p_j.
$$
Next, using the fact that
$
\sum_{j=0}^m (1-j) p_j = 0,
$
we have
$$
\sum_{j=0}^m \frac{j(j-1)}{2}p_j
= \sum_{j=0}^m \left(\frac{(j-1)^2}{2}p_j - \frac{1}{2}(1-j)p_j\right) =
\sum_{j=0}^m \frac{(j-1)^2}{2}p_j.
$$
From \eqref{eq:variance}, we have
$$
\sum_{j=0}^m \frac{(j-1)^2}{2}p_j = \frac{\sigma^2}{2}.
$$
This allows us to conclude that 
$r \geq 1 + \sqrt{2\sigma^{-2}\varepsilon} = 1 + \sigma^{-1}\sqrt{2\varepsilon}.$

\end{proof}

\begin{lemma} \label{lemma:root-asymptotics}
Let $z = 1 + \varepsilon$ for $\varepsilon > 0$.  Then, the characteristic polynomial $Q_z(r)$ defined in  \eqref{eq:charpolyintro} for a probability vector $p = (p_0,\ldots,p_m)$ that satisfies \eqref{meanzero} with $p_1 = 0$ has roots $r_1$ and $r_2$ satisfying
$$r_1 = 1 + \sigma^{-1}\sqrt{2\varepsilon} + \mathcal{O}(\varepsilon), \quad \text{and} \quad r_2 = 1 - \sigma^{-1} \sqrt{2\varepsilon} + \mathcal{O}(\varepsilon),
$$
as $\varepsilon \to 0$,  where $\sigma$ is given by \eqref{eq:variance}.  
\end{lemma}
\begin{proof}[Proof of Lemma \ref{lemma:root-asymptotics}]
Recall from \eqref{eq:conditionforrootqn} that  $Q_z(r)$ has $r$ as a root exactly when 
$$
z = p_0 r  + \sum_{j=2}^m p_j r^{1 - j}.
$$
Let $z = 1+\varepsilon$ and $r = 1 + \delta$. Expanding $z$ in a Taylor series centered at $r=1$ using the Taylor coefficients computed in \eqref{taylora0}, \eqref{taylora1}, and \eqref{taylora2}, gives
$$
1+\varepsilon = 1  + \frac{\sigma^2}{2} \delta^2 + \mathcal{O}(\delta^3).
$$
Solving for $\delta$ gives
$\delta = \pm \sigma^{-1} \sqrt{2 \varepsilon} + \mathcal{O}(\varepsilon).$ It follows that the polynomial $Q_{1+\varepsilon}$ has roots
$$
r_1 = 1 + \sigma^{-1} \sqrt{2 \varepsilon} + \mathcal{O}(\varepsilon), \quad \text{and} \quad 
r_2 = 1 - \sigma^{-1} \sqrt{2 \varepsilon} + \mathcal{O}(\varepsilon), 
$$
which completes the proof.
\end{proof}

Before establishing a fast growth rate in Lemma \ref{lem:growth-actually-happens}, we establish a trivial growth rate, which will be useful in the proof of the fast growth rate.

\begin{lemma} \label{lem:trivialgrowthrate}
Let $P_n(z)$, $n \ge 0$ be the family of polynomials defined by \eqref{recurrence} for a probability vector $p = (p_0,\ldots,p_m)$ that satisfies \eqref{meanzero} with $p_1 = 0$. Then,
$$
P_n(1 + \varepsilon) \ge (1 + \varepsilon)^n, \quad \text{for all} \quad \varepsilon \ge 0 \quad \text{and} \quad n \ge 0.
$$
\end{lemma}
\begin{proof}
Let $\varepsilon \ge 0$ be arbitrary. We start by proving
$P_{n+1}(1+\varepsilon) \ge (1+\varepsilon) P_n(1+\varepsilon)$ for all $n \ge 0$  by strong induction.

For the base cases $0 \leq k \leq m-1$, we have $P_k(1 + \varepsilon) = (1 + \varepsilon)^k$, so the result holds trivially. For the inductive step, let $n \ge m-1$ be given, and assume
$P_{k+1}(1+\varepsilon) \ge (1+\varepsilon) P_k(1+\varepsilon)$ for $n-1 \ge k \ge 0$. The recurrence \eqref{recurrence} then gives us
$$P_{n+1}(1 + \varepsilon) = \frac{(1 + \varepsilon)}{p_0} P_n(1 + \varepsilon) - \sum_{j=2}^m \frac{p_j}{p_0} P_{n+1-j}(1 + \varepsilon).$$
By our inductive hypothesis, we have $P_{n+1-j}(1 + \varepsilon) \le P_{n}(1 + \varepsilon)$ for all $2 \le j \le m$.  As a result, 
\begin{align*}
    P_{n+1}(1 + \varepsilon) &\ge \frac{(1 + \varepsilon)}{p_0} P_n(1 + \varepsilon) - \sum_{j=2}^m \frac{p_j}{p_0} P_{n}(1 + \varepsilon) \\
    &= \frac{1}{p_0}\left(1 + \varepsilon - \sum_{j=2}^m p_j\right)P_n(1+\varepsilon) \\
    &= \frac{1}{p_0}\left(1 + \varepsilon - ( 1 - p_0) \right)P_n(1 + \varepsilon) 
    \\
    &= \left(1 +\frac{\varepsilon}{p_0}\right) P_n(1 + \varepsilon) \\ 
    &\ge (1 + \varepsilon)P_n(1 + \varepsilon), 
\end{align*}
where we used $p_0 + \sum_{j=2}^m p_j = 1$, and the last inequality holds because $p_0 < 1$. Since $P_0(1+\varepsilon) =1$, the inequality
$P_{n+1}(1 + \varepsilon) \ge (1+\varepsilon) P_n(1+\varepsilon)$ and a straightforward induction argument complete the proof.
\end{proof}

\begin{lemma} \label{lem:growth-actually-happens}
Let $P_n(z)$, $n \ge 0$ be the family of polynomials defined by \eqref{recurrence} for a probability vector $p = (p_0,\ldots,p_m)$ that satisfies \eqref{meanzero} with $p_1 = 0$. Then 
\begin{align}\label{eqn:growth-happens}
P_n(1 + \varepsilon)  \ge c (1 + \sigma^{-1} \sqrt{2\varepsilon})^n, \quad  \text{for all} \quad \varepsilon \ge 0 \quad \text{and}  \quad n \ge 0,
\end{align}
where $c > 0$ is a constant that only depends on $p$, and the parameter $\sigma$ is defined in \eqref{eq:variance}.
\end{lemma}
\begin{proof}
By Lemma \ref{lem:trivialgrowthrate} we have $P_n(1+\varepsilon) \ge (1+\varepsilon)^n$ for all $\varepsilon \ge 0$. Since
$$
1 + \sigma^{-1} \sqrt{2 \varepsilon} \le 1+\varepsilon \quad \text{when} \quad \varepsilon \ge 2 \sigma^{-2},
$$
it remains to prove that
$$
P_n(1+\varepsilon) \ge c(1 + \sigma^{-1} \sqrt{2 \varepsilon})^n, \quad \text{for all} \quad 
2 \sigma^{-2} \ge \varepsilon \ge 0 \quad \text{and} \quad n \ge 0,
$$
for some constant $c > 0$ that only depends on $p$.
By Lemma \ref{lemma:root-asymptotics},
the characteristic polynomial 
$Q_{1+\varepsilon}(r)$
defined in \eqref{eq:charpolyintro} has roots
$$
r_1 = 1 + \sigma^{-1} \sqrt{2 \varepsilon} + \mathcal{O}(\varepsilon)
\quad \text{and} \quad
r_2 = 1 - \sigma^{-1} \sqrt{2 \varepsilon} + \mathcal{O}(\varepsilon).
$$
The rest of the spectrum is separated from these roots: there exists $1 > \rho > 0$ such that the remaining roots have magnitude at most $\rho$ for small $\varepsilon$. We have
\begin{equation}\label{eq:recurrence-closed-form}
    P_n(1+\varepsilon) = c_1(\varepsilon) r_1(\varepsilon)^n +
    c_2(\varepsilon) r_2(\varepsilon)^n + 
    R_n(\varepsilon),
\end{equation}
where $c_1(\varepsilon),c_2(\varepsilon)$ are coefficients and $R_n(\varepsilon)$ is a remainder term corresponding to the separated roots, which is an analytic function of $\varepsilon$, see for example \cite[Ch. VII \S 1.3]{kato1966perturbation}. The initial condition $P_0(1+\varepsilon) = 1$ implies
$$
c_1(\varepsilon) + c_2(\varepsilon) = 1 + \mathcal{O}(\varepsilon),
$$
while the initial condition $P_1(1+\varepsilon) = 1+\varepsilon$ gives
$$
c_1(\varepsilon) (1+\sigma^{-1} \sqrt{2\varepsilon}) + c_2(\varepsilon) 
(1-\sigma^{-1} \sqrt{2\varepsilon})
 = 1+ \mathcal{O}(\varepsilon),
$$
where we used the fact that the coefficients are bounded and
$R_0(\varepsilon)$ and $R_1(\varepsilon)$ are $ \mathcal{O}(\varepsilon)$.
Solving these equations gives
$$
c_1(\varepsilon) = 1/2 + \mathcal{O}(\sqrt{\varepsilon}), \quad \text{and} \quad c_2(\varepsilon) = 1/2 + \mathcal{O}(\sqrt{\varepsilon}).
$$
Since $r_1$ is a root of $Q_{1+\varepsilon}$ of multiplicity 1 for $\varepsilon > 0$,  it follows that $c_1(\varepsilon)$ is continuous on $(0, 2\sigma^{-2}]$.  

We have from Lemma \ref{lem:trivialgrowthrate} that $P_n(1+\varepsilon) \ge (1+\varepsilon)^n$ and from Lemma \ref{lem:roughlocation} that $r_1(\varepsilon)$ is the only root of magnitude greater than $1$, which together imply
 that $c_1(\varepsilon) > 0$ on the interval $(0,2\sigma^{-2}]$.  Combining this with $\lim_{\varepsilon \to 0}c_1(\varepsilon) = 1/2$, we conclude that $c_1(\varepsilon) \geq c$ on the interval $[0,2\sigma^{-2}]$ for some $c > 0$.

To complete the proof, it suffices to show that
\begin{equation} \label{eq:niceinequality}
P_n(1+\varepsilon) \ge c_1(\varepsilon) r_1(\varepsilon)^n,
\quad \text{for all} \quad 2 \sigma^{-2} \ge \varepsilon > 0 \quad \text{and} \quad n \ge 0.
\end{equation}
Indeed, we have already proven that $c_1(\varepsilon) \ge c$ for some $c > 0$ on $(0,2\sigma^{-2}]$ and by Lemma 
\ref{thm:root-estimate},
$$
r_1(\varepsilon) \ge 1 + \sigma^{-1} \sqrt{2 \varepsilon},
$$
so in combination, this gives
$$
P_n(1+\varepsilon) \ge c(1+\sigma^{-1}\sqrt{2\varepsilon})^n
\quad \text{for all} \quad 2 \sigma^{-2} \ge \varepsilon > 0 \quad \text{and} \quad n \ge 0.
$$
In the following, we prove that
\eqref{eq:niceinequality} holds.
Since $r_1(\varepsilon)$ is the only root of magnitude greater than $1$, we have
$$
\lim_{n \to \infty} \frac{P_n(1+\varepsilon)}{r_1(\varepsilon)^n} = c_1(\varepsilon).
$$
To establish \eqref{eq:niceinequality}, it suffices to show that for each fixed $2 \sigma^{-2} \ge \varepsilon > 0$, the sequence $P_n(1+\varepsilon)/r_1(\varepsilon)^n$ is a monotone decreasing sequence in $n$ as this implies
$$
\frac{P_n(1+\varepsilon)}{r_1(\varepsilon)^n} \ge c_1(\varepsilon), \quad \text{for all} \quad n \ge 0
$$
which is equivalent to \eqref{eq:niceinequality}. Fix $\varepsilon  > 0$ and set 
$$
a_n = \frac{P_n(1+\varepsilon)}{r_1(\varepsilon)^n}.
$$
We prove that $a_n$ is monotone decreasing by strong induction. For the base case $m-1 \ge k \ge 0$ we have
$$
a_k = \frac{(1+\varepsilon)^k}{r_1(\varepsilon)^k},
$$
which is monotone decreasing since $r_1 \ge 1+\sigma^{-1} \sqrt{2 \varepsilon} \ge 1+\varepsilon$ on the interval $(0,2\sigma^{-2}]$. Then, for the inductive case, assume that $a_0 \ge \cdots \ge a_k$ for $k \ge m-1$. By \eqref{eqn:rerecur} we have
$$
\sum_{j=0}^{m} p_j P_{k+1-j}(z) = z P_k(z).
$$
Dividing both sides of the recurrence by $r_1^{k+1}$ gives
$$
\sum_{j=0}^{m} p_j \frac{P_{k+1-j}(z)}{r_1^{k+1}} = z \frac{P_k(z)}{r_1^{k+1}},
$$
which can be written as
$$
\sum_{j=0}^{m} p_j r_1^{-j} \frac{P_{k+1-j}(z)}{r_1^{k+1-j}} = z r^{-1}_1 \frac{P_k(z)}{r_1^{k}}.
$$
Using the definition of $a_k$ we can rewrite this expression as
$$
\sum_{j=0}^{m} p_j r_1^{-j} a_{k+1-j} = z r^{-1}_1 a_k.
$$
Solving for $a_{k+1}$ gives
$$
a_{k+1}   = \frac{z r^{-1}_1}{p_0} a_k - \sum_{j=2}^{m} \frac{p_j r_1^{-j}}{p_0} a_{k+1-j}. 
$$
We know from \eqref{eq:conditionforrootqn} that for $z=1+\varepsilon$, we have
\begin{equation*} 
z = p_0 r_1  + \sum_{j=2}^m p_j r^{1 - j}_1,
\end{equation*}
which can be rearranged as
$$
1 = \frac{r^{-1}_1 z}{p_0} - \sum_{j=2}^m \frac{p_j r^{-j}_1}{p_0}.
$$
It follows that
\begin{align*}
a_{k+1}   & = \frac{z r^{-1}_1}{p_0} a_k - \sum_{j=2}^{m} \frac{p_j r_1^{-j}}{p_0} a_{k+1-j} \\
&\le  \frac{z r^{-1}_1}{p_0} a_k - \sum_{j=2}^{m} \frac{p_j r_1^{-j}}{p_0} a_k \\
&=  a_k \left( \frac{z r^{-1}_1}{p_0}  - \sum_{j=2}^{m} \frac{p_j r_1^{-j}}{p_0} \right) \\
& = a_k,
\end{align*}
where the inequality uses the inductive hypothesis, that $a_0 \ge \cdots \ge a_k$. Thus, we have established that $a_n$ is monotone decreasing and the proof is complete.
\end{proof}

\begin{lemma} \label{lem:locationrepeatedroots}
Suppose that $p = (p_0,\ldots,p_m)$ is a probability vector
satisfying the mean-zero condition \eqref{meanzero} with $p_1 = 0$ and $p_j >0$ for some $j \ge 3$. Let $Q_z(r)$ be the characteristic polynomial 
\eqref{eq:charpolyintro}, and $z(t)$ be the stability curve \eqref{eq:curveformula}.
Define
$$
S = \left\{t = 2\pi k/n : k\in \{0,\ldots,n-1\} \right\}, 
$$
where $n$ is the greatest common divisor $n = \gcd \{j \in \{2,\ldots,m\} : p_j > 0\}$. Then, the polynomial $Q_{z(t)}$ has $e^{i t}$ as a root of multiplicity $1$ when $t \in [0,2\pi) \setminus S$ and $e^{i t}$ as a repeated root when $t \in S$.
\end{lemma}
\begin{proof}
By the definition of $z(t)$, the polynomial $Q_{z(t)}$ has $e^{i t}$ as a root, so it remains to determine if the root is repeated. If $e^{i t}$ is a repeated root of $Q_{z(t)}$, then $Q'_{z(t)}(e^{i t}) = 0.$ For $w \in \mathbb{C}$, we have
$$
Q'_{w}(r) = m r^{m-1} - (m-1) \frac{w}{p_0} r^{m-2} + \sum_{j=2}^m (m-j) \frac{p_j}{p_0} r^{m-j-1}.
$$
Setting $Q'_w(e^{it}) = 0$ and solving for $w = w(t)$ gives 
$$
w(t) = \sum_{j=0}^m \frac{m-j}{m-1} p_j e^{i(1-j) t}, \quad \text{for} \quad t \in [0,2\pi],
$$
which is a curve such that $Q'_{w(t)}$ has a root $e^{i t}$, where we used the assumption that $p_1 = 0$ to simplify notation. The values of $t$ such that $z(t) = w(t)$ are where $Q_{z(t)}$ has a repeated root $e^{i t}$. Setting $z(t) = w(t)$ gives
$$
\sum_{j=0}^m p_j e^{i(1-j)t} 
=\sum_{j=0}^m \frac{m-j}{m-1} p_j e^{i(1-j) t}. 
$$
Subtracting the right-hand side from the left-hand side gives
\begin{equation} \label{eq:conditionwz}
\frac{1}{m-1} \sum_{j=0}^m (1-j) p_j e^{i(1-j) t} = 0.
\end{equation}
Observe that
$$
z'(t) = i \sum_{j=0}^m (1-j) p_j e^{i(1-j) t}.
$$
Thus, \eqref{eq:conditionwz} implies that $z'(t) = 0$, since the expression in this equation is a constant multiple of $z'(t)$.
By Lemma \ref{locationcusps} we have $z'(t) = 0$ 
if and only if 
$t = 2\pi k/n$ for $k\in \{0,\ldots,n-1\}$, where $n$ is the greatest common divisor 
$$
n = \gcd \{j \in \{2,\ldots,m\} : p_j > 0\},
$$
which completes the proof.
\end{proof}

Finally, we prove boundedness.

\begin{lemma} \label{lem:showboundedness}
Let $P_n(z)$, $n \ge 0$, be the family of polynomials defined by \eqref{recurrence} for a probability vector $p = (p_0,\ldots,p_m)$ that satisfies \eqref{meanzero} with $p_1 = 0$ and $p_j >0$ for some $j \ge 3$.
Let $\Gamma$ denote the closure of the region enclosed by the curve $\gamma$ parameterized by \eqref{eq:curveformula}. Then,
$$
|P_n(z)| \le C \quad  \text{for all} \quad z \in \Gamma \quad \text{and} \quad n \ge 0.
$$
\end{lemma}
\begin{proof}[Proof of Lemma \ref{lem:showboundedness}]
By the maximum modulus principle, it suffices to show that 
\begin{equation} \label{eq:boundedcurve}
|P_n(z(t))| \le C \quad \text{for all} \quad t \in [0,2\pi],
\quad \text{and} \quad n \ge 0,
\end{equation}
where
$$
z(t) = \sum_{j=0}^m p_j e^{i(1-j)t}, \quad \text{for} \quad t \in [0,2\pi],
$$
is the boundary curve of $\Gamma$. 
Note that $P_n(z(t))$ is defined by a recurrence with characteristic polynomial \eqref{eq:charpolyintro}
$$
Q_{z(t)}(r) = r^m - \frac{z(t)}{p_0} r^{m-1} + \sum_{j=2}^m \frac{p_j}{p_0} r^{m-j}.
$$

By Lemma \ref{lem:roughlocation}, all of the roots of $Q_z$ have magnitude strictly less than $1$ when $z \in \interior(\Gamma)$. Since the roots of $Q_z$ are continuous functions of $z$, it follows that all of the roots of $Q_{z(t)}$ have magnitude at most $1$. Recall that $Q_{z(t)}$ has a root $e^{i t}$.  Define
$$
S = \left\{t = 2\pi k/q : k\in \{0,\ldots,q-1\} \right\}, 
$$
where $q$ is the greatest common divisor $q = \gcd \{j \in \{2,\ldots,m\} : p_j > 0\}$. We proceed by considering two cases for $t \in [0,2\pi)$.
Let $\delta > 0$ be a fixed parameter to be chosen later.

\noindent \emph{Case 1:} Suppose $\dist(t,S) \ge \delta$. Then, by Lemma \ref{lem:locationrepeatedroots}, the polynomial $Q_{z(t)}$ has $e^{i t}$ as a root of multiplicity $1$, and the other roots have magnitude strictly less than $\rho$ for some $0 < \rho < 1$. Here, we can write
$$
P_n(z(t)) = c_1(t) e^{i t n} + R_n(t),
$$
where $c_1(t) e^{i t n}$ is the contribution of the leading root, and $R_n(t)$ is the contribution of the roots whose magnitude is at most $\rho < 1$.
Since the roots corresponding to these terms are separated, $c_1(t)$ and $R_n(t)$ are analytic functions of $t$, see for example \cite[Ch. VII \S 1.3]{kato1966perturbation}. Moreover, since $R_n(t)$ represents the contribution of roots of magnitude strictly less than $\rho < 1$, it follows that
$$
|P_n(z(t))| \to |c_1(t)| \quad \text{as} \quad n \to \infty,
$$
uniformly on the set of $t$ such that $\dist(t,S) \ge \delta$.
Since the function $c_1(t)$ is analytic, and the set is compact, it follows that there exists a constant $C_1 > 0$ such that
$$
|P_n(z(t))| \le C_1 \quad \text{for all} \quad t \in [0,2\pi): \dist(t,S) \ge \delta \quad \text{and} \quad n \ge 0.
$$

\noindent \emph{Case 2:} Suppose that $\dist(t,S) < \delta$. By symmetry, suppose that $|t| < \delta$. When $t = 0$, we know that $Q_{z(t)}$ has a repeated root. In the following, we study how this root splits near $0$.
By \eqref{eq:cuspformula}, the Taylor expansion of $z(t)$ near $0$ is
$$
z(t) =  1 - \frac{\sigma^2}{2} t^2 + \mathcal{O}(t^3).
$$
Moreover, from \eqref{taylora0}, \eqref{taylora1}, and \eqref{taylora2}, we have a Taylor series of $z(t)$ as a function of the root variable $r$ at $r=1$
$$
z(t) = 1 + \frac{\sigma^2}{2} (1-r)^2 + \mathcal{O}(1-r)^3.
$$
Equating these Taylor expansions gives
$$
 1 - \frac{\sigma^2}{2} t^2 = 1 + \frac{\sigma^2}{2} (1-r)^2 + \mathcal{O}(t^3).
$$
Solving for $r$ in terms of $t$ gives two roots
$$
r_1(t) = 1 + i t + \mathcal{O}(t^2) \quad \text{and} \quad r_2(t) = 1 - it + \mathcal{O}(t^2).
$$
Note that this is consistent with the fact that $Q_{z(t)}$
has $e^{i t} = 1 + i t + \mathcal{O}(t^2)$ as a root. 
Since we know that all roots except $e^{i t}$ have magnitude less than $1$, it follows that the root $1 - it + \mathcal{O}(t^2)$ has magnitude less than $1$. In summary, we have
$$
P_n(z(t)) = c_1(t) r_1(t)^n + c_2(t) r_2(t)^n  + R_n(t),
$$
where $r_1(t) = e^{i t} = 1 + i t + \mathcal{O}(t^2)$ and $r_2(t) = 1 - i t + \mathcal{O}(t^2)$. The term $R_n(t)$ represents the contribution of the other roots whose magnitude is strictly less than $\rho$ for some $\rho < 1$ and is analytic in $t$. Thus, this term converges to $0$ as $n \to \infty$, uniformly for $|t|$  sufficiently small. Moreover, since $P_n(z(0)) = 1$ for all $n  \ge 0$, it follows that $R_n(0)=0$ for all $n \ge 0$.  Hence, $R_0(t) = \mathcal{O}(t)$ and $R_1(t) = \mathcal{O}(t)$. Thus,
$$
1 = P_0(z(t)) = c_1(t) + c_2(t) + \mathcal{O}(t),
$$
and
$$
z(t) = P_1(z(t)) = c_1(t) r_1(t) + c_2(t) r_2(t) + \mathcal{O}(t).
$$
Since $z(t) = 1+ \mathcal{O}(t^2)$, we have a system of two equations
$$
\begin{pmatrix}
1 & 1 \\
r_1(t) &  r_2(t)
\end{pmatrix}
\begin{pmatrix}
c_1(t) \\
c_2(t) 
\end{pmatrix} = 
\begin{pmatrix}
1 + \mathcal{O}(t) \\
1 + \mathcal{O}(t)
\end{pmatrix}
$$
whose solution is
\begin{equation} \label{eq:linearsysforc1c2}
\begin{pmatrix}
c_1(t) \\
c_2(t) 
\end{pmatrix}
=
\frac{1}{r_2(t) - r_1(t)}
\begin{pmatrix}
r_2(t) & -1 \\
-r_1(t) & 1
\end{pmatrix}
\begin{pmatrix}
1 + \mathcal{O}(t) \\
1 + \mathcal{O}(t)
\end{pmatrix}.
\end{equation}
That is,
\begin{equation} \label{eq:solutionsystemc1c2}
c_1(t) = \frac{r_2(t) - 1 + \mathcal{O}(t)}{r_2(t) - r_1(t)} \quad \text{and} \quad c_2(t) = \frac{1 - r_1(t) + \mathcal{O}(t)}{r_2(t) - r_1(t)}.
\end{equation}
Since $r_1(t) = 1 + i t + \mathcal{O}(t^2)$ and $r_2(t) = 1 - it + \mathcal{O}(t^2)$, we have
\begin{equation} \label{eq:keyestimatesc1c2}
r_2(t) - 1 = \mathcal{O}(t), \quad 1-r_1(t) = \mathcal{O}(t), \quad \text{and} \quad r_2(t) - r_1(t) = -2i t + \mathcal{O}(t^2).
\end{equation}
Using the estimates from  \eqref{eq:keyestimatesc1c2}
to simplify \eqref{eq:solutionsystemc1c2} gives $c_1(t), c_2(t) = \mathcal{O}(1)$. It follows that there exists $C_2 > 0$ such that 
$$
|P_n(z(t))| \le C_2 \quad \text{for all} \quad t : |t| \le \delta \quad \text{and} \quad n \ge 0.
$$
Taking the maximum of the constants in the two cases completes the proof.
\end{proof}

\subsection{Proof of Theorem \ref{thm:approxzn}}
\label{sec:proofthmapproxzn} The proof strategy for Theorem \ref{thm:approxzn} is similar to the method used to establish \eqref{eq:approxxn} outlined in Section \ref{subsec:bg}. However, as in the generalization of this argument used to establish \cite[Theorem 3]{cowal2025faber}, we construct a martingale process instead of using a sum of random variables to handle the initial conditions. Note that the family of polynomials $P_n(z)$ was constructed to have the key property \eqref{eq:expectationprop}, which enables the construction of the martingale process and its concentration by
 Azuma's inequality \cite{Azuma1967}.
 
\begin{lemma}[Azuma's inequality \cite{Azuma1967}] \label{azumaineq}
Let $X_1,X_2,\ldots$ be random variables that satisfy
$$
\mathbb{E}(X_k | X_1,\ldots,X_{k-1}) = 0 \quad \text{and} \quad |X_k| \le a_k,
$$
for $k = 1,\ldots,n$ and $a = (a_1,\ldots,a_n)^\top \in \mathbb{R}^n$. Then for any $t > 0$ we have
$$
\mathbb{P}\left( \sum_{k=1}^n X_k \ge t \right) \le \exp \left( - \frac{t^2}{2\|a\|_2^2} \right).
$$
\end{lemma}

\begin{proof}[Proof of Theorem \ref{thm:approxzn}]
Let $p = (p_0,\ldots,p_m)$ be a probability vector satisfying \eqref{meanzero} with $p_1 = 0$, and $P_n(z)$, $n \ge 0$ be the family of polynomials defined by  \eqref{recurrence}
for these probabilities. Extend the definition of $P_n(z)$ to negative $n$ by $P_{-n} = P_{n}$ for $n < 0$. 

We define a Markov martingale process $Y_0,Y_1,Y_2,\ldots$ on the integers $\mathbb{Z}$ as follows. First, initialize $Y_0 = 0$, and next for $n \ge 1$ we define $Y_n$ by the following three cases:

\noindent \emph{Case 1:} If $Y_{n-1} = k$ for $k \in \{0,1,\ldots,m-2\}$, then we define
$$
\mathbb{P}(Y_n = k+1 | Y_{n-1} = k) = \frac{2k+1}{2k+2}, \quad \text{and} \quad 
\mathbb{P}(Y_n = -(k+1) | Y_{n-1} = k) = \frac{1}{2k+2}.
$$
\noindent \emph{Case 2:} If $Y_{n-1} = k$ for $k \ge m-1$, then we define
$$
\mathbb{P}(Y_n = k+1-j | Y_{n-1} = k ) = p_j, \quad \text{for} \quad j \in \{0,\ldots,m\}.
$$
\noindent \emph{Case 3:} If $Y_{n-1} = k $ for $k < 0$, then we define
$$
\mathbb{P}(Y_n = j | Y_{n-1} =k) = \mathbb{P}(Y_n = -j | Y_{n-1} = -k), \quad \text{for} \quad j \in \mathbb{Z},
$$
such that the transition probabilities are symmetric about zero.

By construction, $(Y_n)$ is a Markov process. We claim that $(Y_n)$ is also a martingale process. Indeed, let $X_n = Y_n - Y_{n-1}$ denote the martingale difference process. We proceed by considering the three cases:

\noindent \emph{Case 1:} When $Y_{n-1} = k$ for $k \in \{0,\ldots,m-2\}$ we have
$$
\mathbb{E}(X_n | Y_{n-1} =k) = (+1) \frac{2k+1}{2k+2} + (-2k-1) \frac{1}{2k+2} = 0.
$$

\noindent \emph{Case 2:} When $Y_{n-1} \ge m-1$, we have
$$
\mathbb{E}(X_n | Y_{n-1} = k) = \mathbb{E}Y = 0,
$$
where $Y$ is an independent random variable satisfying $\mathbb{P}(Y = 1-j) = p_j$ for $j \in \{0,\ldots,m\}$, which has mean-zero due to the assumption that \eqref{meanzero} holds. 

\noindent \emph{Case 3:}  When $Y_{n-1} < 0$, the claim  holds by symmetry. 

Thus, we have established that $Y_n$ is a martingale. By construction, the  martingale differences are bounded by $2m-1$, so 
we can use Azuma's inequality to deduce that
\begin{equation} \label{eq:ynconcentrate}
\mathbb{P}(|Y_n| \ge t) \le 2 e^{-c t^2/n},
\end{equation}
where the constant $c = 1/(2(2m-1)^2)$. Note that this constant can be improved, see Remark \ref{rmk:improveconst}. Also note that the martingale process $(Y_n)$ was defined such that
$$
\mathbb{E} P_{Y_n}(z) = z^n, \quad \text{for} \quad n \ge 0.
$$
We verify this by induction. In the base case, when $Y_0 = 0$, we have
$$
\mathbb{E} P_{Y_0}(z) = P_0(z) = 1 = z^0.
$$
Next, to establish the inductive case assume that $\mathbb{E} P_{Y_{n-1}}(z) = z^{n-1}$. By the tower property of conditional expectation 
$$
\mathbb{E} P_{Y_n}(z) = \mathbb{E} (\mathbb{E}( P_{Y_n}(z) | Y_{n-1})).
$$
We consider again three cases:

\noindent \emph{Case 1:} If $Y_{n-1} = k $ for $k \in \{0,1,\ldots,m-2\}$, then
$$
\mathbb{E}( P_{Y_n}(z) | Y_{n-1}) = \frac{2k+1}{2k+2} P_{k+1}(z) + \frac{1}{2k+2} P_{-k-1}(z).
$$
Taking the expectation of both sides, and using the fact that $P_{k+1}(z) = z P_k(z)$ for $ k \in \{0,1,\ldots,m-2\}$, $P_{-k}(z) = P_{k}(z)$ for $k < 0$, and the inductive hypothesis gives
$$
\mathbb{E}(\mathbb{E}( P_{Y_n}(z) | Y_{n-1})) = z^n.
$$
\noindent \emph{Case 2:} If $Y_{n-1} = k$ for $k \ge m-1$, then by definition $Y_n = Y_{n-1} + Y$ for an independent random variable $Y$ satisfying  $\mathbb{P}(Y = 1-j) = p_j$ for $j \in \{0,\ldots,m\}$. Hence
$$
\mathbb{E}( P_{Y_n}(z) | Y_{n-1}) 
= \mathbb{E}( P_{Y_{n-1}+Y}(z) | Y_{n-1}) = z P_{Y_{n-1}}(z),
$$
by \eqref{eq:expectationprop}. Taking the full expectation and using the inductive hypothesis gives
$$
\mathbb{E}(\mathbb{E}( P_{Y_n}(z) | Y_{n-1})) = z^n.
$$
\noindent \emph{Case 3:} If $Y_{n-1} = k$ for $k < 0$, then the result follows by the symmetry of the definition of $P_n$ and the symmetry of the definition of the transition probabilities of $Y_n$ about $0$.

In summary, we have established that
$$
z^n = \mathbb{E} P_{Y_n}(z) = \sum_{k \in \mathbb{Z}} \mathbb{P}(Y_n = k) P_k(z) = \sum_{k=0}^{\infty} \alpha_k P_k(z),
$$
where $\alpha_k = \mathbb{P}(|Y_n| = k)$ and the final equality follows by symmetry.
Hence,
$$
\left| \sum_{k = 0}^{\lfloor t \sqrt{n} \rfloor} \alpha_k P_k(z) - z^n \right| 
=
\left| \sum^\infty_{k = \lfloor t \sqrt{n} \rfloor +1} \alpha_k P_k(z) \right|.
$$
By Theorem \ref{thm:boundedrapid}, we know that $|P_n(z)| \le C$ on the region $\Gamma$, which is the closure of the region enclosed by the curve  \eqref{eq:curveformula}.
Thus, by the triangle inequality
$$
\left| \sum^\infty_{k = \lfloor t \sqrt{n} \rfloor+1} \alpha_k P_k(z) \right| \le 
C \left| \sum^\infty_{k = \lfloor t \sqrt{n} \rfloor+1} \alpha_k \right|  \le 2 C e^{-c t^2}, 
\quad \text{for} \quad z \in \Gamma,
$$
where the final inequality uses \eqref{eq:ynconcentrate}. 
In summary, we have shown
\begin{equation} \label{eq:finalestprob}
\left| \sum_{k = 0}^{\lfloor t \sqrt{n} \rfloor} \alpha_k P_k(z) - z^n \right| \le C e^{-c t^2},
\quad \text{for} \quad z \in \Gamma,
\end{equation}
for absolute constants $C,c > 0$ that only depend on $p$, where 
$\alpha_k = \mathbb{P}(|Y_n|= k)$, and the proof is complete.
\end{proof}

\begin{remark}[Improving the constant in the exponential in Theorem \ref{thm:approxzn}] \label{rmk:improveconst}

The constant in the exponent in \eqref{eq:finalestprob}
can be improved. 
Azuma's inequality requires an almost sure bound $a_k$ on the magnitude of the $k$-th step, which can lead to pessimistic bounds on random walks that only occasionally take larger steps. In such cases, Freedman's inequality \cite{freedman1975tail}, see the statement in \cite{tropp2011freedman}, can often give sharper bounds. In particular, when $n$ is large, if the variance of the martingale difference process converges to $\sigma^2$, then by Freedman's inequality the constant approaches $1/(2 \sigma^2)$.
\end{remark}

\subsection{Proof of Theorem \ref{ellipseslow}} \label{proofellipseslow}
In this section, we prove Theorem \ref{ellipseslow}.  To provide context, we state several related results underlying the proof of the theorem, building up to Lemma \ref{thmboundellipse}, which establishes Theorem \ref{ellipseslow} after a change of variables.

Throughout this section, we use the fact that, by the maximum modulus principle, the maximum modulus of a polynomial on a curve is equal to the maximum modulus of a polynomial in the region enclosed by the curve.

\subsubsection{Fastest possible growth of polynomials bounded on a unit disk}
Among all polynomials of degree at most $n$ that are bounded by $1$ on the unit disk, the monomial $z^n$ grows the fastest outside the unit disk.
Let $\mathcal{P}_n$ denote the set of polynomials of degree at most $n$.

\begin{lemma} \label{maximalzn}
Suppose that $P \in \mathcal{P}_n$  satisfies $|P(z)| \le 1$ for all $z \in \mathbb{C} : |z| =1$. Then,
$$
|P(z)| \le |z^n| \quad \text{for all} \quad z \in \mathbb{C} :  |z| \ge 1.
$$
\end{lemma}

The result can be proved by dividing $P(z)$ by $z^n$, performing the change of variables $z \mapsto 1/z$, and using the maximum modulus principle. The result of Lemma \ref{maximalzn} can be reformulated as presented in
\cite[Lemma 4.3]{saad2011numerical} by instead normalizing the polynomial to take the value $1$ at a point $\gamma$ outside a disk of radius $\rho$ and then asking how small the maximum value on the disk of radius $\rho$ can be made.

\begin{corollary}[Saad {\cite[Lemma 4.3]{saad2011numerical}}]  \label{corznmax}

Let $|\gamma| \ge \rho > 0$. Then
$$
\left| \frac{\rho}{\gamma} \right|^n = 
\min_{P \in \mathcal{P}_n : P(\gamma) = 1} \max_{|z| = \rho} |P(z)|.
$$
Moreover, the minimum is achieved by $(z/\gamma)^n$.
\end{corollary}

\subsubsection{Fastest growth of polynomials bounded on an ellipse}
Define the conformal map 
$\psi : \mathbb{C} \setminus \{ z \in \mathbb{C} : |z| \le 1\} \to \mathbb{C} \setminus [-1,1]$ by
\begin{equation} \label{confmap}
\psi(r) = \frac{r + r^{-1}}{2}.
\end{equation}
By definition, $\psi^{-1}(z)$ is the larger magnitude zero of $r^2 - 2 r z + 1 =0$, which is the zero outside the complex unit disk. For fixed $\rho > 0$, the image of the circle $C_\rho = \{ z \in \mathbb{C}: |z| = \rho\}$ under the map $\psi$ is an ellipse $E_\rho = \psi(C_\rho)$, which intersects the real and imaginary axes at
$$
E_\rho \cap \mathbb{R} = \left\{\pm \frac{\rho + \rho^{-1}}{2}\right\},
\quad \text{and} \quad
E_\rho \cap i \mathbb{R} = \left\{\pm i \frac{\rho - \rho^{-1}}{2} \right\},
$$
which are the vertices and co-vertices of the ellipse, respectively. The following result due to Saad \cite[Theorem 4.9]{saad2011numerical} considers an analogous optimization problem to Corollary \ref{corznmax} for the case of an ellipse, and gives upper and lower bounds.

\begin{lemma}[Saad {\cite[Theorem 4.9]{saad2011numerical}}]
\label{ellipsemax}
Let $\rho > 1$, and $\gamma \in \mathbb{C}$ be any point not enclosed by $E_\rho$. Then, 
$$
\frac{\rho^n}{|r_\gamma|^n} \le \min_{P \in \mathcal{P}_n : P(\gamma) = 1} \max_{z \in E_\rho} |P(z)| \le \frac{\rho^n + \rho^{-n}}{|r_\gamma^n + r_\gamma^{-n}|},
$$
where $r_\gamma = \psi^{-1}(\gamma)$.
\end{lemma}

The result of Lemma \ref{ellipsemax} can be equivalently formulated as follows.

\begin{corollary} \label{cor2}
Fix $\rho > 1$. Let $\gamma$ be a point in $\mathbb{C}$ that is not enclosed by $E_\rho$. Then,
$$
\frac{|r_\gamma^n + r_\gamma^{-n}|}{\rho^n + \rho^{-n}}
\le \max_{P \in \mathcal{P}_n : \max_{z \in E_\rho} |P(z)|=1\ } |P(\gamma)| \le 
\frac{|r_\gamma|^n}{\rho^n},
$$
where $r_\gamma = \psi^{-1}(\gamma)$.
\end{corollary}

In the following, we state a rescaled version of Corollary \ref{cor2} for the ellipse
\begin{equation} \label{eq:deffrho}
F_\rho := \left\{ \frac{2}{\rho + \rho^{-1}} z \in \mathbb{C} : z \in E_\rho \right\},
\end{equation}
which has vertices $\pm 1$ and co-vertices $\pm i (\rho - \rho^{-1})/(\rho + \rho^{-1})$.

\begin{corollary} \label{cor3}
Fix $\rho > 1$. Let $\gamma$ be a point in $\mathbb{C}$ that is not enclosed by $F_\rho$ defined in  \eqref{eq:deffrho}. Then,
$$
\frac{|s_\gamma^n + s_\gamma^{-n}|}{\rho^n + \rho^{-n}} \le
\max_{P \in \mathcal{P}_n : \max_{z \in F_\rho} |P(z)|=1\ } |P(\gamma)| \le 
\frac{|s_\gamma|^n}{\rho^n},
$$
where $s_\gamma = \psi^{-1}\left( \gamma \frac{\rho + \rho^{-1}}{2} \right)$.
\end{corollary}

We are now prepared to prove the key result of this section, Lemma \ref{thmboundellipse}, from which Theorem \ref{ellipseslow} follows by a change of variables. The main idea of the proof is to set $\gamma = 1+\varepsilon$, apply Corollary \ref{cor3}, and estimate an upper bound.

\begin{lemma}  \label{thmboundellipse}
Fix $1 > \varepsilon > 0$, $\rho > 1$, and $n \in \mathbb{N}$. Suppose that $P \in \mathcal{P}_n$ satisfies $|P(z)| \le 1$ for $z \in F_\rho$, where $F_\rho$ is defined in  \eqref{eq:deffrho}.  Then,
$$
|P(1+\varepsilon)| \le \left( 1 + \frac{3(\rho^2 +1)}{2(\rho^2 - 1)} \varepsilon \right)^n.
$$
\end{lemma}

\begin{proof}[Proof of Lemma \ref{thmboundellipse}]
By Corollary \ref{cor3}, we have
$$
\max_{P \in \mathcal{P}_n : \max_{z \in F_\rho} |P(z)|=1\ } |P(\gamma)| \le 
\frac{|s_\gamma|^n}{\rho^n}.
$$
By definition,
\begin{equation}  \label{eq:sgammarho}
\frac{s_\gamma}{\rho} 
= \frac{(1+\varepsilon)\frac{\rho + \rho^{-1}}{2} + 
\sqrt{\left((1+\varepsilon)\frac{\rho + \rho^{-1}}{2} \right)^2 - 1}
}{\rho}.
\end{equation}
Observe that
$$
\sqrt{\left((1+\varepsilon)\frac{\rho + \rho^{-1}}{2} \right)^2 - 1} = \frac{\rho - \rho^{-1}}{2} \sqrt{1 + (2 \varepsilon + \varepsilon^2) \frac{(\rho + \rho^{-1})^2}{(\rho - \rho^{-1})^2}}.
$$
Recall that
$
\sqrt{1+x} \le 1 + {x}/{2}$, for $x > 0$.
Hence
$$
\sqrt{\left((1+\varepsilon)\frac{\rho + \rho^{-1}}{2} \right)^2 - 1} \le 
\frac{\rho - \rho^{-1}}{2} + \frac{2 \varepsilon + \varepsilon^2}{2} \frac{(\rho + \rho^{-1})^2}{2(\rho - \rho^{-1})}.
$$
Using this inequality, we deduce from \eqref{eq:sgammarho} that
\begin{align}
\frac{s_\gamma}{\rho} 
&= \frac{(1+\varepsilon)\frac{\rho + \rho^{-1}}{2} + 
\sqrt{\left((1+\varepsilon)\frac{\rho + \rho^{-1}}{2} \right)^2 - 1}
}{\rho} \notag \\
&\le 
\frac{(1+\varepsilon)\frac{\rho + \rho^{-1}}{2} + 
\frac{\rho - \rho^{-1}}{2} + \frac{2 \varepsilon + \varepsilon^2}{2} \frac{(\rho + \rho^{-1})^2}{2(\rho - \rho^{-1})}
}{\rho} \notag \\
&= 1 + \frac{\rho^2 + 1}{\rho^2 -1} \varepsilon + \frac{(\rho^2 +1)^2}{4 \rho^2 (\rho^2 -1)} \varepsilon^2, \label{eq:finaleq}
\end{align}
where the final equality follows from rearranging terms. Since we assumed $1 > \varepsilon > 0$ and $\rho  > 1$, we have
$$
\frac{\rho^2 +1}{4 \rho^2} \varepsilon \le \frac{1}{2}.
$$
Combining this estimate with \eqref{eq:finaleq}, we conclude that
$$
\frac{|s_\gamma|}{\rho} \le 1 + \frac{3(\rho^2+1)}{2(\rho^2-1)} \varepsilon,
$$
which completes the proof. 
\end{proof}

\begin{remark}
Note that Theorem \ref{ellipseslow} can be seen as a direct consequence of Lemma \ref{thmboundellipse} by choosing $\rho > 1$ such that
$$
\delta = \frac{\rho - \rho^{-1}}{\rho + \rho^{-1}} = \frac{\rho^2 - 1}{\rho^2 + 1},
$$
which is possible since we assumed $1 > \delta > 0$.
\end{remark}

\subsection{Proof of Theorem \ref{thm:general-momentum-convergence}} \label{sec:proofthm:general-momentum-convergence}
\begin{proof}[Proof of  Theorem \ref{thm:general-momentum-convergence}]
From Lemma \ref{lem:general-momentum-works}, performing $N$ iterations of Algorithm \ref{alg:generalized-momentum} computes 
$\vect{w}_N = P_N(\vect{A}/\lambda_*)\vect{v}_0$.  
Expressing $\vect{v}_0$ in an eigenbasis of $\vec{A}$ 
yields
\[
\vect{w}_N =  P_N(\vect{A}/\lambda_*)\left( \sum_{j = 1}^n a_j\vect{\phi}_j\right) =  \sum_{j = 1}^n P_N(\vect{A}/\lambda_*)a_j\vect{\phi}_j,
\] 
where each $\vect{\phi}_j$ satisfies 
$P_N(\vect{A}/\lambda_*)\vect{\phi}_j = P_N(\lambda_j/\lambda_*)\vect{\phi}_j.$  Thus,
\begin{align}\label{eqn:th21-wN0}
\vect{w}_N = \sum_{j = 1}^n P_N(\lambda_j/\lambda_*)a_j\vect{\phi}_j.
\end{align}
Since $\lambda_1/\lambda_* > 1$ and $\lambda_j/\lambda_* \in \Gamma$ for all $j \geq 2$, Theorem \ref{thm:boundedrapid} guarantees the existence of constants $C$ and $c$ such that 
\[
P_N(\lambda_1/\lambda_*) \geq c\left(1 + \sigma^{-1}\sqrt{2 \varepsilon_\ast}\right)^N,
~\text{ where }~ \varepsilon_* = 
\frac{\lambda_1}{\lambda_*} - 1,
\]
and $|P_N(\lambda_j/\lambda_*)| \leq C \text{ for } j \geq 2.$
Applying this to \eqref{eqn:th21-wN0} we have
\begin{equation}\label{eqn:th21-leftright}
\vect{w}_N = {a_1P_N(\lambda_1/\lambda_*)}
\left( \vect{\phi}_1 + R_N \right), ~\text{ with } R_N =  \sum_{j=2}^n \frac{a_jP_N(\lambda_j/\lambda_*)}{a_1P_N(\lambda_1/\lambda_*)}\vect{\phi}_j.
\end{equation}
The remainder term $R_N$ satisfies
\begin{align}\label{eqn:th21-e}
    \|R_N\|_2
    \leq \frac{1}{|P_N(\lambda_1/\lambda_*)|}\sum_{j = 2}^n \left| \frac{a_j}{a_1} \right| \cdot \left| P_N(\lambda_j/\lambda_*) \right| \cdot\|\vect{\phi}_j\|_2 
     \leq \widehat C 
    \left( 1 + \sigma^{-1}\sqrt{2\varepsilon_\ast}\right)^{-N},
    \end{align}
with $\widehat C = c^{-1}C \sum_{j=2}^n \left| {a_j}/{a_1} \right|.$
Here, $C$ and $c$ are the respective boundedness and rapid growth constants from Theorem \ref{thm:boundedrapid}, which depend only on $(p_0, \dots, p_m)$, and each ratio $|a_j/a_1|$ depends on the initial iterate $\vect{v}_0$.
The bound on the remainder term \eqref{eqn:th21-e} shows that $\vect{w}_N$ aligns with the dominant eigenvector $\vec{\phi}_1$ as $N$ increases. Applying \eqref{eqn:th21-leftright} and \eqref{eqn:th21-e} to 
$\vec{x}_N = \vec{w}_N/\|\vec{w}_N\|_2$ yields \eqref{eqn:thm21}.
\end{proof}

\section{Conclusion}\label{sec:conc}
In this paper, we constructed a family of polynomials $P_k(z)$ related to the Faber polynomials and defined by recurrence relations \eqref{recurrence}, based on mean-zero random walks. We established how these polynomials can be used to approximate $z^n$ by a polynomial of degree $\sim \sqrt{n}$ in associated radially convex domains in the complex plane, referred to as the stability region for that polynomial. We showed that each constructed sequence of polynomials remains bounded within its stability region and has a rapid growth property outside of it. 

Based on the recursive polynomial construction together with the boundedness and rapid growth properties, we developed an acceleration method for the power iteration that applies to a wide class of nonsymmetric matrices.  Our theory shows the method converges to the dominant eigenmode for problems where the possibly complex subdominant eigenvalues are restricted to the polynomial's stability region. This extends the momentum accelerated methods based on Chebyshev polynomials as in \cite{austin2024dynamically,xu2018accelerated} to a range of nonsymmetric problems. We further show this method can outperform Chebyshev-based methods for problems with complex eigenvalues. Additionally, we demonstrated a dynamic implementation of the method similar to that in \cite{austin2024dynamically,cowal2025faber} that achieves optimal convergence without explicit knowledge of $\lambda_2$, which the optimal parameters of the static version of the method depend upon. We expect the ideas and methods introduced in this paper extend to applications beyond the simple power iteration, and may be useful in other iterative methods in numerical linear algebra for nonsymmetric problems.

\bibliographystyle{plain}
\bibliography{refs}

\end{document}